\documentclass[12pt,oneside,a4paper]{amsart}
\usepackage{amscd,a4,amssymb}
\usepackage{typearea,hyperref}

\newcounter{lemmacounter}
\newcounter{thmcounter}

\newcounter{propcounter}

\newtheorem{lemma}[lemmacounter]{Lemma}
\newtheorem{proposition}[propcounter]{Proposition}

\newtheorem{theorem}[thmcounter]{Theorem}

\numberwithin{lemmacounter}{section}
\numberwithin{remcounter}{section}
\numberwithin{propcounter}{section}

\newcommand{\IP}{{\bf P}}
\newcommand{\IC}{{\bf C}}
\newcommand{\IR}{{\bf R}}
\newcommand{\IQ}{{\bf Q}}
\newcommand{\IQbar}{\overline{\bf Q}}
\newcommand{\IZ}{{\bf Z}}
\newcommand{\IN}{{\bf N}}
\newcommand{\IA}{{\bf A}}

\newcommand{\HeightS}{H}
\newcommand{\Height}[1]{\HeightS{({#1})}}
\newcommand{\heightS}{h}
\newcommand{\height}[1]{\heightS{({#1})}}

\newcommand{\IRan}{{\bf R}_{\rm an}}

\newcommand{\alg}[1]{#1^{\rm alg}}
\newcommand{\N}[1]{N({#1})}

\newcommand{\ssm}{\smallsetminus}

\newcommand{\tors}[1]{#1_{\rm tors}}

\newcommand{\remtor}[1]{#1^{\star}}

\newcommand{\red}{{\rm red}}
\newcommand{\ord}{{\rm ord}}

\newcommand{\E}{\mathcal{E}}
\newcommand{\EL}{\mathcal{E}_L}
\newcommand{\XL}{X}

\newcommand{\ta}[1]{#1^{\rm ta}}
\newcommand{\sta}[1]{#1^{\rm sta}}

\renewcommand{\textsl}[1]{\tikz[baseline=(X.base)]
  \node[xslant=0.26] (X) {#1};} 

\newcommand{\qaq}{\quad\text{and}\quad}

\newcommand{\Dset}{W}
\subjclass[2010]{Primary: 14H52; Secondary: 14G40,  11G05, 11U09.}

\begin{document}
\title{ Torsion Points  on   Elliptic Curves in
Weierstrass Form}
\author{P.~Habegger}
\maketitle


\begin{abstract}
  We prove that there are only finitely many complex  numbers $a$ and $b$
with $4a^3+27b^2\not=0$ such that 
the three points $(1,*),(2,*),$ and
$(3,*)$ are simultaneously torsion on the elliptic curve
defined in Weierstrass form by $y^2=x^3+ax+b$.
This gives an affirmative answer to a question raised by Masser and
Zannier.
We thus confirm a special case in two dimensions of the relative  
Manin-Mumford Conjecture  formulated by Pink
and Masser-Zannier. 
\end{abstract}

\section{Main Result}
\label{sec:main}

In pursuit of unlikely intersections, Masser and Zannier 
\cite{MZ:torsionanomalous, MZ:torsionanomalousAJ} proved that there
 are only finitely many complex $\lambda\not=0,1$ such that
\begin{equation}
\label{eq:MZfinite}
    (2,  \sqrt{2(2-\lambda)})\quad\text{and}\quad (3, \sqrt{6(3-\lambda)})
\end{equation}
are  torsion points on the elliptic curve 
given in Legendre form  $y^2 = x(x-1)(x-\lambda)$.


This result provides evidence for  far-reaching conjectures stated by
its authors \cite{MZ:torsionanomalousAJ,MZ:legendre} 
and by Pink \cite{Pink}. 
Both conjectures govern the
distribution of torsion points 
on a subvariety of a  family of abelian varieties
and may be regarded as a relative version of the Manin-Mumford
Conjecture.  They deal with
   unlikely or anomalous intersections
 emphasized in the earlier work of
 Zilber  \cite{Zilber} for constant semiabelian varieties.
In Masser and Zannier's result the subvariety is an algebraic curve inside
the fibered square of the Legendre family of elliptic curves. 

Another natural family of elliptic curves  is the
Weierstrass family. 
Here an elliptic curve is given as $y^2=x^3+ax+b$ where $a$ and $b$ are
complex parameters that satisfy  the 
inequality $4a^3+27b^2\not=0$ to rule out singularities. 
In this context 
Masser and Zannier \cite{MZ:torsionanomalousAJ} asked if a similar
finiteness statement as above holds.  Because there are 
two parameters, the conjectures suggest
 imposing a torsion condition on a
third point to  expect finiteness.

Our main result gives a positive answer to
Masser and Zannier's question and provides the first evidence supporting
a relative Manin-Mumford Conjecture over a base of
 dimension greater than one.

\begin{theorem}
\label{thm:main}
  There are only finitely many complex pairs $(a,b)$ 
with $4a^3+27b^2\not=0$ such that  
\begin{equation*}
  (1,  \sqrt{1+a+b}),\quad
  (2,  \sqrt{8+2a+b}),\qaq
(3,  \sqrt{27+3a+b})
\end{equation*}
are  torsion points on the elliptic curve given in Weierstrass form
 $y^2=x^3+ax+b$.
\end{theorem}

Although the methods we present are as a whole confined to a specific example,
some intermediate steps hold in greater generality. It is therefore
convenient to work in a more general language.
When not stated otherwise, a variety is defined over $\IC$.
We also identify a variety with the set of its complex points. 
If a variety $X$ is defined
over a field $K$  it is sometimes still useful to write $X(K)$ for the 
$K$-rational points on  $X$.

We proceed by reformulating our main result.
Let $S$
be the affine algebraic surface  
\begin{equation}
\label{def:S}
 \{(a,b)\in \IA^2;\,\,
4a^3+27b^2\not=0\};
\end{equation}
it is defined over $\IQbar$, the algebraic closure of $\IQ$ in $\IC$.
The Weierstrass family of elliptic curves
\begin{equation*}
  \E = \left\{([x:y:z],(a,b))\in\IP^2\times S;\,\,
y^2z = x^3+axz^2+bz^3\right\}
\end{equation*}
 is an abelian scheme over the two-dimensional base $S$.
Let $\E^3$ be the three-fold fibered power of $\E$ over $S$
and $\pi:\E^3\rightarrow S$ the structure morphism. We obtain
 an abelian scheme over  $S$.
A complex  point of an abelian scheme  that is torsion in its respective fiber
will be called a torsion point. 

In this language, our result states that all torsion points on
 a certain, explicitly given,
algebraic surface $X\subset \E^3$
are contained in  finitely many fibers of $\E^3\rightarrow S$.
This surface, we call it the $123$-surface,
 is the Zariski closure 
of the affine subset of $\E^3$ 
where the first coordinate in each copy of $\E$ is fixed to be
$1,2,$ and $3$, respectively.
The restriction of $\E^3\rightarrow S$ to $X$ has finite
fibers, so our main result is equivalent to the statement that $X$ contains
only finitely many torsion points. 

The general conjecture stated by Masser and Zannier
 \cite{MZ:torsionanomalousAJ}
expects the torsion points on our surface to
 lie on  finitely many proper abelian subschemes of $\E^3$.
If true, it could at best imply that 
  torsion points do not lie 
 Zariski dense on $X$.
Our Theorem \ref{thm:main} however, is unconditional. Moreover, our
finiteness statement is stronger than the conjecture's conclusion.
This feature is  due to the specific nature of our surface. 

Let us 
consider for the moment a variation of the $123$-surface.
We claim that
there are infinitely many complex $(a,b)\in S$ such that
\begin{equation}
\label{eq:surf}
  (0,\sqrt{b}), \quad (1,\sqrt{1+a+b}), \qaq (-1,\sqrt{-1-a+b})
\end{equation}
are torsion points on the elliptic curve $y^2=x^3+ax+b$.
Indeed, we find them on $b=0$. The first point is automatically
torsion of order $2$. We observe that $y^2=x^3+ax$ yields an elliptic
curve with complex multiplication and $j$-invariant $1728$. It follows
from basic
facts on elliptic curves that there are infinitely many  $a\in\IC\ssm\{0\}$
such that $(1,\sqrt{1+a})$ is torsion on $y^2=x^3+ax$;
we shall prove a related statement in Lemma \ref{lem:noanocurve}.
We fix such an
$a$. Then $(-1,\sqrt{-1-a})$ is the image of $(1,\sqrt{1+a})$ under
 an automorphism of order $4$ of $y^2=x^3+ax$. So all three points in
 (\ref{eq:surf}) are torsion.
Using a specialization argument
one can show 
 that the algebraic surface in $\E^3$ induced by (\ref{eq:surf}) is
not in a proper abelian subscheme of $\E^3$. Conjecturally, it does
not contain a Zariski dense set of torsion points. 

Let us briefly recap the proof of Theorem \ref{thm:main}. It
splits up into two parts. In the first half,
laid out in Section \ref{sec:legendre}, we
work in the Legendre family of elliptic curves 
\begin{equation*}
  \EL = \left\{([x:y:z],\lambda)\in \IP^2 \times Y(2);\,\,
y^2z = x(x-z)(x-\lambda z) \right\} 
\end{equation*}
where $Y(2)=\IP^1\ssm\{0,1,\infty\}$.
 The three-fold fibered power of
$\EL\rightarrow Y(2)$ is denoted by $\EL^3$.
Working in the Legendre family has the  advantage that the
base is one dimensional.

Any elliptic curve over $\IC$ is isomorphic to an elliptic
curve in Legendre form.
Using a base change argument we can
construct a new algebraic surface in $\EL^3$  using the $123$-surface.
The study  of torsion points on the $123$-surface will be carried
out by studying torsion points on this new surface. 

The first part of the proof makes no use of the
special form of the $123$-surface.
So all partial results will be   formulated for an arbitrary
irreducible algebraic surface $X_L$ in $\EL^3$.

On any elliptic curve, or more generally, on any abelian scheme
 we use $[N]$ to denote the multiplication by
$N\in\IZ$ morphism.
In Proposition \ref{prop:finiteness} we prove  that $X_L$ contains
only finitely many torsion points outside the so-called torsion anomalous locus
of $X_L$.
Informally, this is the union of
 all positive dimension subvarieties of $X_L$ on which an excessive
 number of independent integral relations 
 \begin{equation}
\label{eq:linrel}
   [\alpha](P_1) + [\beta](P_2) + [\gamma](P_3) = 0
\quad\text{where}\quad (P_1,P_2,P_3,\lambda)\in X_L\quad\text{with fixed}
\quad \alpha,\beta,\gamma\in\IZ
 \end{equation}
hold identically.  A precise definition is provided in Section \ref{sec:legendre}.

To prove Proposition \ref{prop:finiteness} we follow the basic strategy,  proposed
originally by Zannier, and  estimate from above and below the number of
 rational points on certain sufficiently tame sets.
This strategy already appeared in the proof of Masser and Zannier's
result mentioned further up and in a new proof of the Manin-Mumford
Conjecture
by Pila and Zannier \cite{PilaZannier}. 

An elliptic logarithm of a torsion point on an elliptic curve
has rational coefficients with respect to a chosen  period
lattice basis.
The conjugate of any torsion point again leads to a rational point.
This observation together with lower bounds for the Galois orbit of a
torsion point yields the required lower bounds for rational points. 
Masser and Zannier required an upper bound, proved by Pila, for the number of rational points
with fixed denominator  on compact subanalytic surfaces. 

Additional difficulties arise in our situation since $X_L$
is an algebraic surface as opposed to the algebraic curve used to
treat (\ref{eq:MZfinite}).
For example, a crucial height inequality  used by Masser and Zannier
which depends on work of Silverman has only recently been extended  to
higher dimension \cite{Hab:Special} by the author.

Algebraic independence statements for certain
transcendental functions related to elliptic logarithms 
 played an important role in 
Masser and Zannier's result regarding (\ref{eq:MZfinite}) and even
more so in their generalization  to  curves \cite{MZ:legendre}.
Their statements, but also the more general result of Bertrand
\cite{Bertrand:90}, cannot be applied directly to the higher
dimensional case.
We  overcome this difficulty using two tools. First, we use
a bound of David
 \cite{DavidPetiteHauteur} on the number of
torsion points defined over
a number field on an elliptic curve.
The quality of his bound is indispensable in our method. It enables
us
to choose a ``wandering curve'' in $X_L$ containing sufficiently many
conjugates of a given torsion point and hence 
makes the use of results from the one dimensional setting feasible.
Second, we replace Pila's counting result by the  powerful
theorem of Pila and Wilkie \cite{PilaWilkie}
 formulated in the versatile language of o-minimal structures.  
This additional generality is required to treat the real $4$ dimensional
sets which arise naturally  in our problem. An equally important
aspect of the Pila-Wilkie Theorem is that it is uniform over
a  definable family. This enables us to manage the  wandering curve
constructed above.

A brief recollection of the theory of o-minimal structures is
presented in Subsection \ref{sec:ominimal}.
 Using David's result we will find an abundance of 
rational points coming from elliptic logarithms
on one fiber of a definable family. Enough actually,
to successfully compete with the upper bound coming from the Pila-Wilkie Theorem.

In the second half of the proof, detailed in Section \ref{sec:notanom},
 we return to the Weierstrass
family, the  natural setting of our main result. 
The obstruction to obtaining finiteness in the first half
was the torsion anomalous locus of $X_L$.
There is also an analogous locus for algebraic surfaces in $\E^3$. 
The goal of  the second half 
is to get  hold of this locus for the $123$-surface. 
In fact, Proposition \ref{prop:noano}  tells us that it is empty.
We briefly indicate the general idea of the argument.

Typically, an anomalous subvariety is an irreducible algebraic curve $C\subset X$ on which
two independent relations as in (\ref{eq:linrel}) hold. 
We can specialize to any point in the image of $C$ under
$\E^3\rightarrow S$.
This yields three points on an elliptic
curve over $\IC$ which are connected by two independent relations. 
In this situation it seems difficult to directly extract information from the fact that the first affine
coordinates of these three points are $1,2,$ and $3$.
Roughly speaking, we will  specialize to a point on the boundary of a
compactification of $S$. In practice we will work 
with $C\rightarrow\pi(C)$ coming from
the restriction to $C$ of $\pi:\E^3\rightarrow S$.
Passing to the generic fiber yields a point on the cube of an elliptic
curve defined
over the function field of $\pi(C)$. 
 Assume for now that this elliptic curve has  a place of bad multiplicative
reduction. 
We can  use the Tate uniformization which relates the group structure  of an
 elliptic curve and the multiplicative group of a field. 
In many instances, this will allow us to
translate the
 excessive number of integral
relations into a multiplicative relation  involving
algebraic numbers derived from 
 $1,2,$ and $3$.
It is then a simple matter to show that this
  multiplicative relation is untenable. 
From this we will deduce that
 the generic fiber of $C\rightarrow \pi(C)$ must have good reduction
 everywhere.
Therefore, all fibers  share a common
$j$-invariant. From this severe
restriction it will not be difficult to derive a contradiction
using  the particular nature of the $123$-surface.

So  we make heavy use of the  special nature of our surface in
the second half.
What happens if one replaces $1,2,3$ by another triple of algebraic numbers?
We have seen that finiteness need not hold even if 
the triple consists of 
 pairwise distinct integers.
In an unpublished manuscript 
the author described a necessary condition on the triple to ensure
a finiteness statement as in Theorem \ref{thm:main}.
For example, the first three primes $2,3,5$ also yield a finiteness
result
as in our main result.


The author is very grateful to David Masser and Umberto Zannier for
the numerous and productive conversations we had in Pisa, July 2010.   
He also thanks the latter for the invitation to Pisa and 
the Scuola Normale Superiore for its hospitality
and financial support. The author was also supported by  SNSF project
number 124737.
 
\section{Torsion Points Outside the Torsion Anomalous Locus}
\label{sec:legendre}

We will work with an irreducible closed algebraic surface $\XL$ in $\EL^3$.
The $123$-surface will not appear in the current
section. So no ambiguity can occur if we avoid the more cumbersome notation $X_L$ from  the
introduction and use $\pi$ to denote the projection $\EL\rightarrow Y(2)$.
We do keep the subscript in $\EL$ to emphasize that 
we are in the Legendre family.

For $\lambda \in Y(2)$ the fiber
  $(\EL)_\lambda = \pi^{-1}(\lambda)$ 
is taken as an elliptic curve given in Legendre form. 
We identify the three-fold fibered power $\EL^3$  
of $\EL\rightarrow Y(2)$ with
\begin{equation*}
  \EL^3 = \{(P_1,P_2,P_3,\lambda) \in (\IP^2)^3\times Y(2);\,\,
P_1,P_2,P_3\in (\EL)_\lambda\}.
\end{equation*}
By abuse of notation we also use $\pi$ for the projection
$\EL^3\rightarrow Y(2)$ and write $(\EL^3)_\lambda =
\pi^{-1}(\lambda)\subset \EL^3$.
Recall that $(P_1,P_2,P_3,\lambda)\in\EL^3$ is called
torsion if  $P_1,P_2,$ and $P_3$ are torsion points of
 $(\EL)_\lambda$.

Any $\chi=(\alpha, \beta,\gamma)\in \IZ^3$ determines a  Zariski closed set
$G_\chi\subset \EL^3$ through the integral relation
\begin{equation*}
[\alpha] (P_1) + [\beta](P_2) + [\gamma]( P_3) = 0.
\end{equation*}

 An irreducible closed subvariety $A$ of $\XL$ is called
a torsion anomalous subvariety of  $\XL$
\begin{enumerate}
\item [(i)] if $\dim A = 1$ and two independent integral relations hold
  on $A$,  
\item[(ii)] or if 
 $\dim A= 2$ and one non-trivial integral relation holds
  on $A$, 
\item[(iii)]
or  if $\dim A \ge 1$ and $A$ is an irreducible component of an algebraic
subgroup of $(\EL^3)_\lambda$ 
for some $\lambda\in Y(2)$ such that $(\EL)_\lambda$ has complex multiplication.
\end{enumerate}
The torsion anomalous locus of $\XL$
is $ \bigcup_A A$, here $A$ runs
over all torsion anomalous subvarieties of $\XL$.
We write $\ta{\XL}$ for the \emph{complement} of the torsion anomalous locus in $\XL$.

An irreducible closed subvariety $A\subset \EL^3$ 
which dominates $Y(2)$
is called a 
 component of flat subgroup scheme of $\EL$ 
 \begin{enumerate}
 \item [(i)] if $\dim A = 1$ and three independent integral relations hold
  on $A$,
\item[(ii)] or if $\dim A = 2$ and two independent integral relation hold
  on $A$,
\item[(iii)] or if
$\dim A = 3$ and one independent integral relation holds
  on $A$,
\item[(iv)] or if $A=\EL^3$.
 \end{enumerate}
We write $\remtor{\XL}$ for $\XL\ssm \bigcup_A A$, here $A$ runs
over all components of flat subgroup schemes of $\EL$ contained
completely in $\XL$. We have $\ta{\XL}\subset \remtor{\XL}$.

The definition of $\remtor{\XL}$ coincides with the complex points of the corresponding
definition given in \cite{Hab:Special}.
Indeed, see Lemma 2.5 in this reference.

The purpose of this section is to prove 
that there are only finitely many points outside the torsion anomalous
locus of $X$.

\begin{proposition}
\label{prop:finiteness}
Let $\XL\subset \EL^3$ be an irreducible closed algebraic surface
defined over $\IQbar$ which dominates $Y(2)$. 
\begin{enumerate}
\item [(i)]
There are at most finitely many torsion points
in $\ta{\XL}$.
\item[(ii)] The set 
  \begin{equation*}
    \left\{\pi(P);\,\, P\in \remtor{\XL}\text{ is torsion and }(\EL)_{\pi(P)}
\text{ has complex multiplication}\right\} 
  \end{equation*}
is finite.
\end{enumerate}
\end{proposition}


It is  conceivable that the union in the definition of 
$X\ssm \ta{\XL}$ or $X\ssm \remtor{\XL}$ is over  infinitely many $A$. 
So we have no reason to
expect that $\ta{\XL}$ or $\remtor{\XL}$ is Zariski open. 
However,  $\remtor{\XL}$ is known to be Zariski open by Theorem 1.3(i)
\cite{Hab:Special}. 
In a later section we will address the problem of describing
$\ta{\XL}$ for an algebraic surface derived from the $123$-surface.
In our situation, $\ta{\XL}$ will be Zariski open. 

We do expect $\ta{\XL}$ to be Zariski open in general. More precisely,
we expect $\XL$ to contain only finitely many torsion anomalous
subvarieties that are not strictly contained in another torsion
anomalous subvariety of $\XL$. 

Let us assume for the moment that this finiteness statement holds for
$\XL$,
 that no non-trivial integral relation holds identically on $\XL$, and that $\XL$
dominates $Y(2)$.
In this case we sketch how  Proposition
\ref{prop:finiteness} implies a uniform
Manin-Mumford-type statement 
in a family of abelian varieties. Indeed, we may regard $\XL$ as a
 family of curves $\{\XL_\lambda = \XL\cap \pi^{-1}(\lambda)\}$ parametrized by $\lambda\in Y(2)$. 
Up-to finitely many exceptions, controlled by the proposition, any
torsion point on a member of this family lies on one of finitely many anomalous
subvariety as in cases (i)
and (iii)  of the definition.
So any torsion point on $X$ satisfies two independent
relations coming from a fixed finite set.
If we are in case (i) then these relations are integral; in case (iii)
they have coefficients in the endomorphism ring of an elliptic curve
with complex multiplication. 
It is not difficult to deduce that
$X_\lambda$ contains a positive dimensional
 irreducible component of an algebraic subgroup
for at most finitely many $\lambda$.
For all other $\lambda$  two independent relations as above
intersect $X_\lambda$ in a finite set whose cardinality  can be bounded
from above independently of $\lambda$ using B\'ezout's Theorem. 
We conclude that after omitting finitely many $\lambda$ there
 is a uniform upper bound for the number of torsion points on $X_\lambda$.

In the remainder of this section we will assume that $\XL$ is 
as in the proposition. So it dominates $Y(2)$ and we may
 fix a number field $F\subset \IQbar$ over which it is defined.

We will work with
real parameters $B\ge 1$ and $\delta \in (0,1]$.
Here $\delta$ may depend on $B$ and $B$ may
depend on the surface $X$ and on $F$.
If not stated otherwise, 
the symbols $c_1,c_2,\dots$ will denote positive
constants which may depend  $\XL$,  $F$,
$\delta$, and $B$. During the proof $B$ and $\delta$ will be chosen properly.

\subsection{o-minimal Structures}
\label{sec:ominimal}

We provide the definition of an o-minimal structure. For an in-depth
treatment
of this subject we refer to van den Dries's book \cite{D:oMin}.

Let $\IN=\{1,2,3,\dots\}$. 
An o-minimal structure is a
sequence $\mathfrak S=(S_1,S_2,\dots)$ 
such that if $n,m\in\IN$ then 
   $S_n$ is a collection of subsets of
$\IR^n$  with the following properties.
\begin{enumerate}
\item [(i)] The
 intersection  of two sets in $S_n$ is in  $S_n$
and the complement of a set in $S_n$ is in $S_n$.
\item[(ii)] Any real semi-algebraic subset of $\IR^n$ is in
  $S_n$.
\item[(iii)]  The Cartesian product of a set in
  $S_n$ with a set in $S_m$ is in $S_{n+m}$.
\item[(iv)] 
The image of a set in $S_{n+m}$
under the projection
 $\IR^n\times\IR^m\rightarrow
  \IR^n$  onto the first $n$ coordinates  is in $S_n$.
\item[(v)] A set in $S_1$ is a  finite union of points and open,
  possibly unbounded,  intervals.
\end{enumerate}

The first four properties assert that an o-minimal structure contains
enough interesting sets to work with.  The fifth property restricts the possible sets in all
$S_n$ because these project to $\IR$ by (iv). 

We call a subset of $\IR^n$ definable in $\mathfrak{S}$ if it lies in $S_n$.
If $X\subset \IR^n$ then we call a function $f:X\rightarrow \IR^m$ 
definable in $\mathfrak{S}$ if its graph, a subset of $\IR^n\times\IR^m$, lies in $S_{n+m}$.
Domain and image of a  function that is definable in $\mathfrak{S}$
are definable in $\mathfrak{S}$.

A subset $Z$ of $\IR^n\times\IR^m$ that is definable in $\mathfrak{S}$
is sometimes called a  family definable in $\mathfrak{S}$.
We do this to  emphasizes  that
$Z$ can be seen  as a collection of subsets of $\IR^n$ parametrized by $\IR^m$.
Concretely, for $y\in\IR^m$ we let $Z_y$ denote the projection
of $Z\cap (\IR^n\times\{y\})$ to $\IR^n$. Then $Z_y$ is definable in $\mathfrak{S}$.

To formulate the result of Pila and Wilkie mentioned in the introduction,
we shall define the exponential Weil 
  height on the rational numbers
 by setting $\HeightS(p/q)=\max\{|p|,q\}$
for coprime integers $p$ and $q$ with $q\ge 1$. 
In higher dimension we set $\Height{\xi_1,\dots,\xi_n} = 
\max\{\Height{\xi_1},\dots,\Height{\xi_n}\}$ for
$(\xi_1,\dots,\xi_n)\in\IQ^n$.
Let $X\subset \IR^n$ be any subset for the moment.
 The counting function
associated to $X$ is
\begin{equation*}
  \N{X,T} = 
\#\{\xi\in X\cap\IQ^n;\,\, \Height{\xi}\le T\}
\quad\text{for}\quad T\ge 1;
\end{equation*}
there are only finitely many points in $\IQ^n$ of bounded height,
so the cardinality is finite.

We define $\alg{X}\subset X$ to be 
 the union of all connected, positive dimensional real
semi-algebraic sets contained in $X$.

\begin{theorem}[Pila-Wilkie \cite{PilaWilkie}]
\label{thm:PW}
  Let $Z\subset\IR^n\times\IR^m$ be a family definable in an o-minimal
  structure and let $\epsilon >0$. There is a constant $c>0$ depending
  on $Z$ and $\epsilon$ such that if $y\in\IR^m$, then
  \begin{equation*}
    \N{Y\ssm \alg{Y},T}\le c T^\epsilon\quad\text{for all}\quad T\ge 1
  \end{equation*}
where $Y=Z_y$.
\end{theorem}

By the Tarski-Seidenberg Theorem,  the collection of
all real semi-algebraic sets satisfies (iv) in the definition of an
o-minimal structure. 
From this it is not difficult to show that the real semi-algebraic
sets  define an o-minimal structure.
But this structure is not large enough for our needs. 
Luckily, a variety of larger o-minimal structures are known. 
For example, van den Dries \cite{Dries:TS} reinterpreted a result of Gabrielov as
stating that the so-called
finitely subanalytic sets form an o-minimal structure $\IRan$. 
We will give not a  definition of such sets. It
suffices to remark 
that the restriction to $[-1,1]^n$ of a real valued analytic function 
on a neighborhood of $[-1,1]^n$ is definable in $\IRan$.
This will be enough functions for our application. 

For the remainder of this section we will call sets, functions, and
families definable if they are
definable in $\IRan$.

We could not find a reference for the following, possibly well-known,
statement. Therefore, we provide its short proof which is valid in any
o-minimal structure.

\begin{lemma}
\label{lem:injective}
  Let $X\subset \IR^n$ be a definable set  and let $f:X\rightarrow
  \IR^m$  be a definable
  function. 
There are definable sets $X_0,\dots,X_M\subset\IR^n$ 
with 
 $X=X_0\cup X_1\cup\cdots \cup
  X_M$
 such that 
$f|_{X_1},\dots, f|_{X_M}$ are injective and
such that the fibers of $f|_{X_0}$ contain
no isolated points. 
Here $X_0=\emptyset$ and $M=0$ are possible.
\end{lemma}
\begin{proof}
We first prove the lemma  when $f$ has finite fibers. Then the
fibers have cardinality bounded from above uniformly by
Corollary 3.6, page 60
\cite{D:oMin}. 
Say $c$ is the maximal cardinality attained.
We may suppose $c\ge 2$. 
By Definable Choice, Proposition 1.2, page 93 {\it ibid.}, 
there is a definable function $g:f(X)\rightarrow \IR^n$
with $f(g(y)) = y$ for all $y\in f(X)$. 
The sets $g(f(X))$ and  $X \ssm g(f(X))$ are definable.
Now the definable function $f|_{g(f(X))}$ is injective and the
 fibers of the definable function $f|_{X\ssm g(f(X))}$  have cardinality at most $c-1$.
The current case of the lemma follows by induction on $c$. 

In the general case we 
observe that
\begin{equation*}
  X_0=\{x\in X;\,\, \dim_x f^{-1}(f(x)) \ge 1\}
\end{equation*}
is a definable set by applying the Cell Decomposition Theorem,
cf. page 62 {\it ibid}. We note that $X_0$ contains no isolated points.
The function $f$ restricted to its complement in $X$ has discrete
fibers.
Again by Corollary 3.6, page 60 {\it ibid.} these fibers are finite.
 This enables us to reduce to the situation above. 
\end{proof}

\subsection{A Definable Family}
\label{sec:ominimal}
In the current subsection, any  reference to  a topology on $\XL$ or
$\EL^3$ will refer to the Euclidean topology if not stated otherwise. 


In a neighborhood of $1/2\in Y(2) =\IC\ssm\{0,1\}$ 
we may describe a period lattice basis  of
the fiber of $\EL$ using  Gauss's hypergeometric
function, cf. Chapter 9 \cite{Husemoeller}. This period lattice basis can be continued analytically along any
path in $Y(2)$.
 We fix a path from any point 
 in $\EL^3$ to the zero element
of $(\EL^3)_{1/2}$.
We  continue the periods along the path induced in $Y(2)$.

Any $P\in \EL^3$ has
 a  neighborhood $V_P$  in $\EL^3$ 
on which we may choose holomorphic elliptic logarithms 
\begin{equation*}
  z_{P1},z_{P2},z_{P3} : V_P\rightarrow \IC.
\end{equation*}
We may also fix holomorphic functions $f_P,g_P:V_P\rightarrow \IC$
whose values determine a basis of the period lattice
of the corresponding fiber. 

The values of $f_P$ and $g_P$
 are $\IR$-linearly
independent. We can express $z_{Pk}$ in terms of $f_P$ and $g_P$
using
real analytic  functions
$\xi_{P1},\dots,\xi_{P6}:V_P\rightarrow \IR$, i.e.
\begin{equation*}
  z_{P1} = \xi_{P1} 
f_P + \xi_{P2} g_P,
 \quad
  z_{P2} = \xi_{P3} 
f_P + \xi_{P4} g_P,
 \quad\text{and}\quad  
z_{P3} = \xi_{P5} 
f_P + \xi_{P6} g_P.
 \quad
\end{equation*}
We write $\theta_P: V_P\rightarrow \IR^6$ for the
real analytic function 
\begin{equation*}
 Q\mapsto (\xi_{P1}(Q),\dots,\xi_{P6}(Q)).
\end{equation*}
It provides coordinates of 
 an elliptic logarithm of $Q$ in terms of
the period lattice basis given by $f_P(Q)$ and $g_P(Q)$.

After shrinking $V_P$ we may suppose that it is
contained in an affine subset of $\EL^3$. 
This has the effect that
if $X'\subset \EL^3$ is Zariski closed then
 $X' \cap V_P$ can be described as the set of common
zeros of finitely many polynomials restricted to  $V_P$.

We note that $\EL^3$ is an $8$-dimension
 real analytic manifold.
After shrinking $V_P$ there is a real bianalytic map
$\vartheta_P:(-2,2)^8{\rightarrow} V_P$ 
taking $0$ to $P$.
We define 
\begin{equation*}
  U_P = X\cap \vartheta_P([-1,1]^8) \subset V_P.
\end{equation*}
Then $U_P$
 is compact since $\vartheta_P([-1,1]^8)$ is compact.
It is also a neighborhood of $P$ in $X$. 

The compact set
\begin{equation*}
  \Lambda_\delta = \{z\in \IC;\,\, \delta\le |z|\le \delta^{-1}\quad
  \text{and}\quad
 |1-z|\ge\delta \}
\end{equation*}
is contained in $Y(2)$.
The pre-image $\pi|_X^{-1}(\Lambda_\delta) = X \cap \left((\IP^2)^3\times
\Lambda_\delta\right)$ is also compact. 
This set is covered by all neighborhoods $U_P$ with
$P \in \pi|_{X}^{-1}(\Lambda_\delta)$.
So there is a positive integer $c_1$ and
 $P_1,\dots, P_{c_1}\in \pi|_X^{-1}(\Lambda_\delta)$ with
$U_{P_1}\cup\cdots\cup U_{P_{c_1}} \supset \pi|_X^{-1}(\Lambda_\delta)$.

In the following, we drop the $P$ and write $U_i,V_i,\theta_i,\vartheta_i$ 
for $U_{P_i},V_{P_i},\theta_{P_i},\vartheta_{P_i}$, respectively.

Let $|\cdot|$ denote the maximum norm on $\IR^n$. 

\begin{lemma}
\label{lem:basic}
  Let $1\le i\le c_1$. There are sets $U_{i0},\dots ,
  U_{i M_i}$ with $U_i = U_{i0}\cup\cdots \cup U_{i M_i}$ such that 
the following properties hold. 
  \begin{enumerate}
\item[(i)]
The functions $\theta_i|_{U_{i1}},\dots,\theta_i|_{U_{iM_i}}$ are injective
and the fibers of $\theta_i|_{U_{i0}}$ contain no isolated points. 
  \item [(ii)] 
If $X'\subset \EL^3$ is Zariski closed, then 
$\theta_i(X'\cap U_{ij}) \subset\IR^6$
is definable for all $0\le j\le M_i$.
\item[(iii)] 
There is $c_2$ with $|\xi| \le c_2$ if $\xi \in \theta_i(U_i)$.
  \end{enumerate}
\end{lemma}
\begin{proof}
By construction, $X\cap V_i$ is the zero set in $V_i$  of
functions that are polynomial on $V_i$.
So each pre-image $\vartheta_i^{-1}(U_i) = 
\vartheta_i^{-1}(X\cap V_i)\cap [-1,1]^8$
is the set of common zeros of finitely many real analytic functions on
$(-2,2)^8$ restricted to  $[-1,1]^8$.
Therefore, it is definable in our o-minimal structure $\IRan$.

Observe that $\theta_i\circ\vartheta_i$ is real analytic on $(-2,2)^8$. 
Its restriction to $\vartheta_i^{-1}(U_i)$ is thus definable. 
We apply  Lemma \ref{lem:injective} to 
$\theta_i\circ\vartheta_i|_{\vartheta_i^{-1}(U_i)}$ and obtain $M+1$
definable subsets of $\vartheta_i^{-1}(U_i)$.
Taking their images under  $\vartheta_i$ gives 
 $U_{i0},U_{i1},\dots,U_{i M_i}$ with $U_i = U_{i0}\cup\cdots\cup U_{i
  M_i}$.
The statement of Lemma \ref{lem:injective} 
and the fact that $\vartheta_i$ is injective and continuous 
is what is needed for (i). 

Let $X'$ be as in part (ii).
As before, $\vartheta_i^{-1}(X'\cap V_i) \cap [-1,1]^8$ is a definable
set
and therefore so is 
$\vartheta_i^{-1}(U_{ij}) \cap \vartheta_i^{-1}(X'\cap V_i)\cap [-1,1]^8 =
\vartheta_i^{-1}(X'\cap U_{ij})$.
Its image $\theta_i(X'\cap U_{ij})$ under the definable function
$\theta_i\circ\vartheta_i|_{[-1,1]^8}$ is definable.
 This shows (ii). 

Part (iii) follows since  $U_i$ is compact and $\theta_i$ is continuous. 
\end{proof}

In order to avoid double indices we rename $U_{ij}$
as $U_i$ by increasing, if necessary, the constant $c_1$.
Of course, we also adjust the $\theta_i$ accordingly. For example, in
this new notation claim (i) of the preceding lemma states 
 that
$\theta_i|_{U_i}$ is either injective or has fibers without isolated
points. 

We define 
\begin{equation*}
  \Dset_i = \theta_i(U_i)\subset \IR^6.
\end{equation*}
This is a definable set by part (ii) of the lemma above
applied to $X\supset U_i$.

The image of a torsion point of order $N$ in $U_i$ lies
in $\frac 1N \IZ^6 \cap \Dset_i$.
For this reason we are interested in the distribution of  rational points on
$\Dset_i$.
Below, we will find many such rational points on a fiber of 
\begin{alignat}1
\label{eq:defZi}
  Z_i = \{ &(\xi_1,\dots,\xi_6, \alpha,\beta,\gamma,\psi,\omega)\in \Dset_i\times
  \IR^5;\\
\nonumber
&\alpha \xi_1 + \beta \xi_3 + \gamma \xi_5 = \psi\quad\text{and}\quad
\alpha \xi_2 + \beta \xi_4 + \gamma \xi_6 = \omega
\} \subset \IR^6\times\IR^5
\end{alignat}
considered as a family parametrized by $\IR^5$.
We note that the $Z_i$ are definable because their definition 
involve only  definable sets and
the  basic algebraic operations.



The next lemma is the theorem of  Pila and Wilkie
adapted to our situation.

\begin{lemma}
\label{lem:pilawilkie}
There exists a positive constant $c_3$, depending on the usual data,
 such that if $1\le i\le c_1$ and
$y\in \IR^5$, then 
\begin{equation*}
  \N{Y\ssm\alg{Y},T}\le c_3 T^{1/12}
\quad\text{for all}\quad T\ge 1
\end{equation*}
where $Y=(Z_i)_y$.
\end{lemma}
\begin{proof}
  This follows from  Theorem \ref{thm:PW} adapted to our situation.
\end{proof}

As we will see below, it is critical that
 this estimate is uniform in the parameter $y$.
We work with the exponent $1/12$  for expository reasons;  the Theorem
of Pila-Wilkie provides any positive $\epsilon$ at the cost
of increasing $c_3$.

\subsection{The Galois Orbit  of a Torsion Point}

Let $E$ be an elliptic curve defined over a number field $K$. It is
well-known that the group of torsion points  $\tors{E(K)}$ of $E(K)$ is finite.
By a deep result of Merel its cardinality $\#\tors{E(K)}$
 is bounded from above  solely in terms of
$[K:\IQ]$. In particular, the bound does not depend on the
 height of $E$.
Our method  
allows us to assume that the height of $E$ is bounded. 
So the deep uniformity aspect  in Merel's work
will not play a role here. On the other hand, our argument
is quite sensitive in the dependency in $[K:\IQ]$ of the bound for
$\#\tors{E(K)}$.

The following result of David is essentially best possible
 with regard to the
degree for an unrestricted elliptic curve.

For a definition and basic properties of the absolute logarithmic Weil
height $\heightS$, or just height for short, we refer to Chapter 1.5 in Bombieri
and Gubler's book \cite{BG}.

\begin{theorem}[David \cite{DavidPetiteHauteur}]
\label{lem:david}
There exists a positive absolute  constant $c_4$ with the following property.
  Let $E$ be an elliptic curve defined over a number field $K$ and
 let $h_0\ge 1$ be a bound for the height of the $j$-invariant
  of $E$.
Then
\begin{equation*}
 \#\tors{E(K)} \le c_4 h_0 [K:\IQ] \log(3[K:\IQ]). 
\end{equation*}
\end{theorem}
\begin{proof}
  This follows from Th\'eor\`eme 1.2(i) \cite{DavidPetiteHauteur}. Indeed, torsion
  points have N\'eron-Tate height zero.
\end{proof}

Our  approach works as long as one has a  bound  of the form
  $\# \tors{E(K)}\le c(h_0)[K:\IQ]^{\kappa}$ 
with fixed $\kappa < 3/2$  and where $c(h_0)$ is allowed to depend on $h_0$.

\subsection{Torsion Points on $\XL$}
\label{sec:torsion}

Throughout this subsection we work with a fixed 
torsion point $P=(P_1,P_2,P_3,\lambda)\in \XL(\IQbar)$.
 We will additionally assume
 \begin{equation*}
 \height{\lambda}\le B; 
 \end{equation*}
 here $B$ is the parameter introduced in
beginning of this section.
It will be fixed at a later point in the proof and may
depend on $\XL$ but not on $P$.
We recall that $\delta,c_1,c_2,\ldots$ may depend on $B$; but they
shall not depend on $P$.

Let $N$ be the order of $P$.
For brevity, say
$K = F(P) \subset \IQbar$  and $D = [K:F]$.
We remark $\lambda \in K\ssm\{0,1\}$.
We write $\Sigma$ for the set of  embeddings $\sigma:K\rightarrow
\IC$ that restrict to the identity on $F$. Then $\#\Sigma =  D$.

\begin{lemma}
\label{lem:constructGa}
There exist a   positive absolute  constant $c_8$ and
 $\chi\in \IZ^3\ssm \{0\}$ with 
\begin{equation*}
\max\{N,|\chi|^3\}\le c_8 D \log (3D)   
\end{equation*}
such that $P \in G_\chi$.
\end{lemma}
\begin{proof}
  The three torsion points $P_1,P_2,P_3$ generate a finite subgroup $\Gamma$ of  
$\tors{(\EL)_\lambda(K)}$.
Being a finite subgroup of an elliptic curve, $\Gamma$ is isomorphic
to 
$(\IZ/N'\IZ)\times (\IZ/ R\IZ)$ for some positive integers $R| N'$.
Since $\Gamma$ is killed by multiplication by $N$ we find
$N'|N$. But we must have $N'=N$ since $P$ has order  $N$.

Finding $\chi=(\alpha,\beta,\gamma)\in \IZ^3\ssm\{0\}$ with $[\alpha](P_1) +
[\beta](P_2)+[\gamma](P_3)
= 0$ on $(\EL)_\lambda$ amounts to finding $(\alpha,\beta,\gamma,*,*)\in
\IZ^5\ssm\{0\}$ 
 in the kernel of  a certain matrix
\begin{equation}
\label{eq:linsystem}
\left[
  \begin{array}{ccccc}
    * & * & * & N & 0 \\
    * & * & * & 0 & R
  \end{array}
\right]
\end{equation}
where the entries denoted by $*$ are integers;
in the first and second row they lie in $[-N/2,N/2]$
and $[-R/2,R/2]$, respectively.

We apply  Siegel's Lemma as stated in Corollary 2.9.7
\cite{BG}.
The height of the system (\ref{eq:linsystem})
is at most $c_5 NR$ with $c_5 > 0$ absolute. 
Since our system has three independent solutions, there is a solution
in $\IZ^5\ssm\{0\}$ with maximum norm at most $c_6(NR)^{1/3}$. 
Forgetting the last two  coordinates gives
\begin{equation}
\label{eq:bounda}
  |\chi| \le c_6 (NR)^{1/3}.
\end{equation}

On the other hand, we have $NR= \# \Gamma \le \# \tors{(\EL)_\lambda(K)}$.
David's result from the last section 
implies $NR\le c_4 h_0 D\log(3D)$, here $h_0$ is $1$ more than  the
height of the $j$-invariant of $(\EL)_\lambda$.
This $j$-invariant equals
\begin{equation}
\label{eq:lambdatoj}
  j = 2^8 \frac{(\lambda^2-\lambda+1)^3}{\lambda^2(\lambda-1)^2}
\end{equation}
by Proposition III 1.7(b) \cite{Silverman:AEC}.
Elementary height inequalities imply that $h_0$ is bounded in
terms of $\height{\lambda}$. So 
$h_0$ is bounded in terms of $B$.
Hence $NR\le c_7 D \log(3D)$
and in particular
 $N \le c_7 D \log(3D)$.
This is the bound for $N$ in the assertion. 
We find  the bound for $|\chi|^3$ by recalling (\ref{eq:bounda}).
\end{proof}



Any embedding $\sigma\in\Sigma$  determines a torsion point
$P^\sigma=\sigma(P)\in \XL(\IQbar)$.

\begin{lemma}
  For $\delta \in (0,1]$ sufficiently small in terms of $B$ and $F$
  there is a positive constant $c_9\le 1$
and an index $1\le i_0\le c_1$ such that for at least $c_9D$
embeddings $\sigma\in \Sigma$ we have
\begin{equation*}
 \pi(P^\sigma)  \in \Lambda_\delta
\qaq P^\sigma \in U_{i_0}.
\end{equation*}
\end{lemma}
\begin{proof}
A similar statement was given in Lemma 6.2
\cite{MZ:torsionanomalousAJ}. Recall $\lambda=\pi(P)\in K$.
Let
 $\delta \in (0,1]$ and let us assume
 $\sigma(\lambda)\not\in\Lambda_\delta$ for 
 more than $D/2$ embeddings $\sigma\in\Sigma$.
Then one of 
\begin{equation*}
|\sigma(\lambda)|>\delta^{-1},\quad
  |\sigma(\lambda)|^{-1}>\delta^{-1}, \quad
|1-\sigma(\lambda)|^{-1}> \delta^{-1}
\end{equation*}
holds for more than $D/6$ embeddings $\sigma\in\Sigma$.

By elementary height properties  we have
$\height{\lambda^{-1}}=\height{\lambda}\le B$ and
$\height{(1-\lambda)^{-1}} = \height{1-\lambda}\le
\height{\lambda}+\log 2 \le B+\log 2$.
  The definition of the height as stated on the bottom of page 16 \cite{BG} implies
\begin{alignat*}1
\height{\lambda}&+\height{\lambda^{-1}}+\height{(1-\lambda)^{-1}} \\
&\ge \frac{1}{[K:\IQ]}
\sum_{\sigma:K\rightarrow\IC}
\log\left(\max\left\{1,|\sigma(\lambda)|\right\}
 \max\left\{1,\frac{1}{|\sigma(\lambda)|}\right\} 
\max\left\{1,\frac{1}{|1-\sigma(\lambda)|}\right\}\right)
\end{alignat*}
here $\sigma$ runs over all embeddings of $K$ into $\IC$.
We bound $3B+\log 2 \ge D/(6[K:\IQ]) \log(\delta^{-1})$.
But $[K:\IQ] = D[F:\IQ]$,
so $\log(\delta^{-1})\le 6[F:\IQ](3B+\log 2)$.

So if $\delta \in (0,1]$ is sufficiently small with
respect to $B$ and $F$ there are least $D/2$ embeddings
$\sigma\in\Sigma$
satisfying $\sigma(\lambda)=\pi(P^\sigma)\in\Lambda_\delta$.
Recall that $\pi|_X^{-1}(\Lambda_\delta)$
is covered by  $U_1,\dots,U_{c_1}$.
The lemma  follows from  the Pigeonhole Principle
on taking $c_9 = 1/(2 c_1)$.
\end{proof}

We fix $\delta$ and $i$ once and for all as in this lemma and let
$\Sigma'\subset \Sigma$ denote the subset provided  therein.
We abbreviate
$U=U_{i_0}$, $\Dset=\Dset_{i_0}$, $Z=Z_{i_0}$, and $\theta=\theta_{i_0}$ from
  Subsection \ref{sec:ominimal}.
The fact that $i_0$ may depend on $P$ will be harmless. 

The conjugates $P^\sigma$ lie in $U$ for all $\sigma\in \Sigma'$.
We denote their images under $\theta$ by
\begin{equation*}
 \xi^\sigma =
 (\xi^\sigma_1,\dots,\xi^\sigma_6) = \theta(P^\sigma)
\in \Dset.
\end{equation*}
Since $P^\sigma$ has order $N$ we have
$\xi^\sigma \in \frac 1N \IZ^6$ for the coordinates in terms of the
period lattice basis.

Before we continue, let us recapitulate the current situation and also
describe how we will proceed.
In total there are $D$ conjugates of $P$ over $F$. Of these, a fixed positive
proportion lies on the set $U\subset X$.  
So by Lemma \ref{lem:constructGa}, the number of conjugates on $U$ is
at least of  order $N/\log N$.
The next lemma is crucial.
It states 
that among the embeddings considered above,   at least approximately $N^{1/3}/\log N$
yield a $\xi^\sigma$ in a fixed fiber of the definable family $Z$
constructed around (\ref{eq:defZi}). 
 We will show that
the number of $\xi^\sigma$ equals the number of conjugates $P^\sigma$,
at least in the most interesting cases. 
As we have seen above, the $\xi^\sigma$ are rational. 
Their heights turn out to be bounded
linearly in terms of $N$. Consequentially, we will have found many rational points of bounded height 
on a fixed fiber of $Z$. 
But we have no
control over the  precise fiber containing these rational points;
its existence is derived from the Pigeonhole Principle.
This is compensated by the fact that the Pila-Wilkie Theorem is 
 uniform over definable families.
We  then  conclude the existence
of a semi-algebraic curve inside a fixed fiber of $Z$.
Such a curve will lead to a torsion anomalous subvariety of $X$.

\begin{lemma}
\label{lem:pigeonhole}
  There exist a positive constant $c_{12}$, a tuple
 $y=(\alpha,\beta,\gamma,*,*)\in \IZ^5$
with $(\alpha,\beta,\gamma)\not=0$, and a
  subset $\Sigma''\subset\Sigma'$ with 
  \begin{equation*}
 \# \Sigma'' \ge c_{12}
  \frac{N^{1/3}}{\log(3N)}   
\quad\text{such that}\quad
\xi^{\sigma} \in Z_{y}
\quad\text{for all}\quad
\sigma\in\Sigma''.
  \end{equation*}
\end{lemma}
\begin{proof}
 Let $\chi=(\alpha,\beta,\gamma)$ be as in Lemma \ref{lem:constructGa}.
Then $P\in G_\chi$ and even $P^\sigma \in G_\chi$ for all
$\sigma\in\Sigma$. 
For $\sigma\in\Sigma'$, the 
 period coordinates  satisfy
  \begin{equation}
\label{eq:lincomb}
    (\alpha \xi^\sigma_1 + \beta \xi^\sigma_3 + \gamma \xi^\sigma_5,
\alpha \xi^\sigma_2 + \beta \xi^\sigma_4 + \gamma \xi^\sigma_6) \in \IZ^2.
  \end{equation}

A simply application of the triangle inequality together with the
bound for $\xi^\sigma_j$ from Lemma \ref{lem:basic}(iii) gives
\begin{equation*}
  |\alpha \xi^\sigma_1 + \beta \xi^\sigma_3 + \gamma \xi^\sigma_5 |
\le 3 c_2 |\chi|.
\end{equation*}
The same bound holds for
 $|\alpha \xi^\sigma_2 + \beta \xi^\sigma_4 + \gamma \xi^\sigma_6 |$.
So the number of possibilities for the integral vector
(\ref{eq:lincomb}) is at most
 $(6c_2|\chi|+1)^2$ as $\sigma$ runs over $\Sigma'$.
Using Lemma \ref{lem:constructGa}, the number of possibilities is at most
 $c_{10} D^{2/3} \log(3D)^{2/3}$.

We recall $\#\Sigma' \ge c_9D$.
By the Pigeonhole Principle
there is a subset $\Sigma''\subset \Sigma'$ with 
\begin{equation*}
  \#\Sigma'' \ge \frac{c_9 D}{c_{10} D^{2/3} \log(3D)^{2/3}}
= c_{11} \left(\frac{D}{\log(3D)^{2}}\right)^{1/3}
\end{equation*}
such that (\ref{eq:lincomb}) attains the same value  for all
$\sigma\in\Sigma''$.
We use elementary
estimates and $N \le c_8 D\log(3D)$ from Lemma \ref{lem:constructGa}   to conclude
\begin{equation*}
  \#\Sigma'' \ge 
c_{11}\left(\frac{D\log(3D)}{\log(3D)^{3}}\right)^{1/3}
\ge
c_{11} \left(\frac{D\log(3D)}{\log(3D\log(3D))^{3}}\right)^{1/3}
\ge c_{12} \frac{N^{1/3}}{\log(3N)}.\qedhere
\end{equation*}
\end{proof}


We recall some notation  from \cite{Hab:Special}.
There $\ker[N]$ was defined as  the kernel of
 the multiplication
by $N$ morphism $[N]:\mathcal{E}_L^3\rightarrow \mathcal{E}_L^3$. 

Next we find a condition which guarantees 
that the
conjugates of $P$ indeed lead to many rational points $\xi^\sigma$.
 The condition is satisfied if for example $P$ is not inside an
anomalous subvariety of $X$.

\begin{lemma}
\label{lem:remtor}
Let us assume that $\{P\}$ is an irreducible component
of $\XL \cap \ker[N]$.
Then $\theta|_U:U\rightarrow \IR^6$ is injective
and in particular,
$\#\{\xi^\sigma;\,\,\sigma\in\Sigma''\}=\#\Sigma''$.
\end{lemma}
\begin{proof}
By Lemma \ref{lem:basic}(i) we know that 
 $\theta|_{U}$ is either injective or has fibers
without isolated points.
Say we are in the second case
and let us fix $\sigma \in \Sigma''$.
The fiber of $\theta$ containing any $P^\sigma$ 
also contains an infinite sequence $(P_k)_{k\in\IN}$
with $P_k\in U\ssm \{P^\sigma\}$
converging to $P^\sigma$. Since elliptic logarithms of 
$P_k$ and $P^\sigma$ have the same
coordinates with respect to a period lattice basis we find $P_k\in
\ker[N]$. Therefore, $\{P^\sigma\}$ is not an irreducible component of
$\XL\cap \ker[N]$. The same holds true for $\{P\}$
and this contradicts our hypothesis.
\end{proof}

We now apply  the Theorem of Pila-Wilkie.

\begin{lemma}
\label{lem:algcurve}
 Assume $P$ satisfies the hypothesis of Lemma \ref{lem:remtor}
and suppose $N$, the order of $P$, is sufficiently large, i.e. $N\ge c_{15}$.
There exist $\chi\in\IZ^3\ssm\{0\}$, an irreducible component
$C\subset \XL \cap G_\chi$, and $\sigma\in\Sigma$
with $P^\sigma \in C$  such that 
$\theta(C\cap U)$ contains a connected real semi-algebraic curve.
\end{lemma}
\begin{proof}
 Let $y=(\alpha,\beta,\gamma,\psi,\omega)\in\IZ^5\ssm\{0\}$ and
 $\Sigma''$ be as provided by  Lemma
 \ref{lem:pigeonhole}. 
Say $\sigma\in\Sigma''$. Then $P^\sigma$ is torsion of order $N$ and
  we have $\xi^\sigma \in \frac{1}{N} \IZ^6$.
On the other hand, $|\xi^\sigma|\le c_2$ by Lemma \ref{lem:basic}(iii).
Therefore,
\begin{equation}
\label{eq:xirational}
  \xi^\sigma \in \IQ^6\quad\text{and}\quad \Height{\xi^\sigma} \le c_{13}N
\quad\text{with}\quad c_{13}=\max\{1,c_2\}.
\end{equation}

We set $T = c_{13} N\ge 1$.
By  Lemma \ref{lem:pigeonhole} we have
$\Sigma'' \ge 
c_{14} T^{1/3-1/6} = c_{14} T^{1/6}$.
The number of rational points $\xi^\sigma$ 
is thus at least $c_{14}T^{1/6}$ by Lemma
\ref{lem:remtor}.
However, the upper bound from Lemma \ref{lem:pilawilkie}
gives
\begin{equation*}
  \N{Z_{y}\ssm\alg{(Z_{y})},T}\le c_3 T^{1/12}.
\end{equation*}
We may assume that $T=c_{13}N$ is sufficiently large to the end that
 $c_{14}T^{1/6}> c_{3}T^{1/12}$. Hence
 there exists
 $\sigma\in \Sigma''$ with
$\xi^\sigma \in \alg{(Z_{y})}$.
In other words, there is a connected real semi-algebraic set $R$ in
$Z_{y}$ of positive dimension that contains $\xi^\sigma$.

Any $\xi' = (\xi_1,\dots,\xi_6)\in Z_y$ satisfies
\begin{equation*}
  \alpha \xi_1+\beta \xi_3 +\gamma \xi_5 = \psi \quad\text{and}\quad
  \alpha \xi_2+\beta \xi_4 +\gamma \xi_6 = \omega.
\end{equation*}
By definition, $Z_{y}\subset \Dset  = \theta(U)$. 
So there is $Q\in U$ with 
$\theta(Q)=\xi'$. The linear relations imply
$Q\in G_\chi$ with $\chi=(\alpha,\beta,\gamma)$. 
We conclude $Z_y\subset \theta(   U\cap G_\chi )$.
 Let $\XL\cap G_\chi = C_1\cup\cdots\cup C_r$ be the decomposition into
 irreducible components. So $Z_y\subset\bigcup_k \theta(C_k \cap U)$. 


Since $Z_y$ contains a connected real semi-algebraic set of positive
dimension that passes through $\xi^\sigma$, it is reasonable to expect some
$\theta(C_k\cap U)$ to do the same. Let us now prove this fact.
By Proposition 3.2, page 100 \cite{D:oMin} 
there is a continuous semi-algebraic function $\gamma:[0,1]\rightarrow Z_y$
 with $\gamma(0) = \xi^\sigma$ and
 $\gamma(1)\not=\gamma(0)$. 
Each $\theta(C_k\cap U)$ is definable by Lemma \ref{lem:basic}(ii).
The pre-images $I_k=\gamma^{-1}( \theta(C_k \cap U))\subset \IR$ are
definable
 and their union is $[0,1]$. Recall that 
$U$ is compact.  So each  $I_k$ is   closed
because $\theta(C_k\cap U) \subset \IR^6$ is closed.
By property (v) of an o-minimal structure,
 each $I_k$ is a finite union of closed  intervals.
So there is $k$ and $t\in (0,1]$ such that
 $I_k$ has $[0,t]$ as a connected
component.
We may choose $k$ such that $t$ is maximal.
So $\gamma|_{[0,t]}$ maps to $\theta(C\cap U)$ with $C=C_k$;
in particular, $\xi^\sigma \in \theta(C\cap U)$.
What if $\gamma|_{[0,t]}$ is constant? Then $t<1$ because
$\gamma(1)\not=\gamma(0)$.
By a similar argument as above, the interval $[t,1]$ can be covered by
pre-images which are themselves finite unions of closed intervals. From
this we deduce a contradiction to the maximality of $t$.
So $\gamma|_{[0,t]}$ is non-constant. Its image is 
a connected real semi-algebraic
curve
 which is completely contained in $\theta(C\cap U)$.

This implies the second assertion of the lemma.
It also shows that $\xi^\sigma = \theta(P')$ for some
$P' \in C\cap U$.
But recall that $\theta|_U$ is injective by Lemma \ref{lem:remtor}
and $\theta(P^\sigma)=\xi^\sigma$. Therefore, $P^\sigma=P'\in C$. 
\end{proof}


\begin{lemma}
\label{prop:one}
  Let $C\subset \EL^3$ be an irreducible algebraic curve such that
$\theta(U\cap C)$ contains a connected real semi-algebraic curve.
  \begin{enumerate}
  \item [(i)]
If $\pi|_C:C\rightarrow Y(2)$ is dominant  there exist
independent $\chi',\chi''\in\IZ^3$ with $C\subset G_{\chi'}\cap G_{\chi''}$.
\item[(ii)]
If $\pi|_C:C\rightarrow Y(2)$ is not dominant, then it is constant and
$C$ is the translate of an algebraic subgroup of $(\EL^3)_{\pi(C)}$.
  \end{enumerate}
\end{lemma}
\begin{proof}
Part (i) follows from Bertrand's Th\'eor\`eme 5 \cite{Bertrand:90} applied to the three
  possible projections of $C$ onto $\EL^2$.  
Alternatively, we can also refer to Masser and Zannier's Appendix A \cite{MZ:legendre}.

Part (ii) is a consequence of  Ax's Theorem 3  \cite{Ax} for a fixed abelian variety.
\end{proof}

\subsection{Proof of Proposition \ref{prop:finiteness}}

We begin by fixing the parameter $B$ used above.

By Theorem 1.3(ii) \cite{Hab:Special} there exists $B\ge 1$,
 depending  on $\XL$, with
$\height{\pi(P)}\le B$ for all torsion points $P\in \remtor{X}\cap X(\IQbar)$.

Let $P \in \remtor{\XL}$ be a torsion point of order $N$ and set $\lambda=\pi(P)$.

The Zariski closed set $\ker[N]$ is equidimensional of dimension 1 by
Lemma 2.5 \cite{Hab:Special}.
So $\{P\}$ is an irreducible component of the intersection $\XL\cap \ker[N]$. 
We can deduce two things. First, using the fact that $\XL$
and $\ker[N]$ are defined over $\IQbar$ we find that $P$ is algebraic, i.e.
$P\in X(\IQbar)$. 
Second, $\height{\lambda}\le B$. So $P$ is as in Subsection \ref{sec:torsion}.

After omitting finitely many $P$ we may suppose that $N$ is sufficiently large; 
for example
 $N\ge c_{15}$, the constant
from Lemma \ref{lem:algcurve}.
We remark that $P$ satisfies the hypothesis of this lemma.

We will prove part (ii) first. So we shall additionally assume that
$(\EL)_\lambda$ has complex multiplication.
The $j$-invariant $J$ of the elliptic curve $(\EL)_\lambda$
is given by (\ref{eq:lambdatoj}). By basic height properties and
$\height{\lambda}\le B$, we find
that $\height{J}$ is  bounded from above independently of $P$.
A result of Poonen \cite{Poonen:MRL01} states that 
the set of $j$-invariants of bounded height coming from elliptic
 curves with complex multiplication is finite. So there are only
 finitely many possibilities for $J$. By (\ref{eq:lambdatoj}) the same
 holds true for
 $\lambda$.

We now prove part (i). 
 We now  assume in addition $P\in \ta{X} \subset\remtor{X}$.

We have already assumed $N$ to be large; this will lead to a
contradiction as follows. 
Let  $\chi\in\IZ^3\ssm\{0\}$ and $C \subset \XL\cap G_{\chi}$
 be as in Lemma \ref{lem:algcurve}.
Then $C\not=\XL$ because otherwise $X\subset G_\chi$  would imply
$\ta{X}=\emptyset$. So
 $\dim C \le 1$.
General intersection theory
implies $\dim C \ge \dim \XL - 1$. Hence $C$ is an algebraic curve defined over $\IQbar$.
Recall that $P^\sigma$  lies on $C$ for some $\sigma\in\Sigma$. 
We split up into cases regarding whether $\pi|_C:C\rightarrow Y(2)$ 
is dominant or not.

First we assume $\pi|_C$ is dominant. By Lemma \ref{prop:one}(i) the algebraic curve
$C$ lies in  $G_{\chi'} \cap G_{\chi''}$
for independent $\chi',\chi''\in\IZ^3$. 
But for an appropriate conjugate $C'$ of $C$ we have
$P\in C'$ and $C'\subset G_{\chi'}\cap G_{\chi''}$.
Therefore, $C'$ is torsion anomalous which contradicts $P\in\ta{X}$.

Now say $\pi|_C$ is not dominant. 
This means that $C$ is contained in a single fiber of
$\EL^3\rightarrow Y(2)$. 
 We know from Lemma
\ref{prop:one}(ii) that $C$ is a
translate of an algebraic subgroup of a fiber of $\EL^3$.
But $C$ contains $P^\sigma$, which is torsion. So $C$ is the translate of an
algebraic subgroup by a torsion point. 
Conjugating,  we find that $P$ is on an algebraic curve $C'$
which is the translate of an
algebraic subgroup of  $(\EL^3)_\lambda$ by a torsion point.

If $(\EL)_\lambda$ does not have  complex multiplication 
then $C'\subset G_{\chi'}\cap  G_{\chi''}$ for independent 
$\chi',\chi''\in\IZ^3$.
This means that $C'$ is a torsion anomalous subvariety of $\XL$ as in
part (i) of the definition.
But $P\in C'$, contradicting our hypothesis $P\in\ta{X}$. 

Finally, suppose $(\EL)_\lambda$ has complex multiplication.
Then is $C'$ a torsion anomalous subvariety as in part (iii) of the definition. As above we
arrive at a contradiction.
\qed

\section{Torsion Anomalous Subvarieties}
\label{sec:notanom}

The results in this section are formulated using the Weierstrass
family of elliptic curves.
Recall that the base $S$ is  the algebraic
surface  given by (\ref{def:S}).
The fiber above $(a,b)\in S$ is an elliptic curve with $j$-invariant
 $j(a,b) = 2^8 3^3 a^3/(4a^3+27b^2)$. We regard $j:S\rightarrow\IA^1$ as
a morphism.

Recall that $\E^3$ is a five-dimensional non-singular irreducible variety.
By abuse of notation, $\pi$ denotes both  structure
morphisms $\E\rightarrow S$ and $\E^3\rightarrow S$.
Both are  proper morphisms.
It is straightforward to check that the $123$-surface $X$ is irreducible.

In Section \ref{sec:legendre} we defined torsion anomalous 
subvarieties of an irreducible algebraic surface in $\EL^3$. 
The analog definition for a surface in
$\E^3$
is somewhat more involved. This is due to the fact that  fibers  of
$\E^3 \rightarrow S$  are isomorphic along algebraic curves in $S$ where $j$ is constant. 
Before coming to the definition we state
an elementary lemma which is used through this section. It
enables us to pass from the Weierstrass to the Tate model of an
elliptic curve.

If $K$ is a field then $K^\times = K\ssm\{0\}$.

\begin{lemma}
\label{lem:isom}
  Let $K$ be a field of characteristic not equal to $2$ or $3$. Say we
  are given two  elliptic curves
  \begin{alignat}1
\nonumber
   E &: y^2 = x^3 + ax + b \quad\text{and}\quad \\
\label{eq:E2}
   E' &: y^2 + xy = x^3+a'x + b' 
  \end{alignat}
with $a,b,a',b'\in K$ that
 are isomorphic over $K$. Then there exists $w\in
 K^\times$  such that 
 \begin{equation*}
   (x,y)\mapsto \left(w^2x - \frac{1}{12},w^3y - \frac 12w^2 x +
   \frac{1}{24}\right)
 \end{equation*}
determines an isomorphism $E\rightarrow E'$ with
\begin{equation}
\label{eq:wab}
  w^4 a = a' - \frac{1}{48}\quad\text{and}\quad
  w^6 b = -\frac{1}{12} a' + b' + \frac{1}{864}.
\end{equation}
Moreover,  any elliptic curve over $K$ is
isomorphic over $K$ to one given as in (\ref{eq:E2}).
\end{lemma}
\begin{proof}
  This follows from the basic theory of elliptic curves \cite{Silverman:AEC}.
\end{proof}

Now we come to the auxiliary construction needed for the definition
of torsion anomalous subvarieties.
Let $A\subset \EL^3$ be an irreducible
closed
subvariety such that $j\circ \pi|_{A}$ is constant with value
$J\in\IC$.
 Then $\pi(A)$ is either a
point or an irreducible algebraic curve. 

We assume the latter for the moment and
set $C=\pi(A)$. We take the coordinates $a$ and $b$ of $S$ as 
elements in the function field $\IC(C)$ of $C$. Then
$4a^3(J-1728)+27b^2 J = 0$.
So $\IC(a,b)$ is a rational
   function field generated by some $t\in \IC(C)$. 
We may assume
\begin{equation}
  \label{eq:ab}
 (a,b) = 
\left\{
\begin{array}{ll}
  (0,t) &:\quad \text{if }J= 0, \\
  (t,0) &:\quad \text{if }J=1728, \\
  (t^2,\zeta t^3) &:\quad \text{if }J\not=0,1728\text{ for some }\zeta\in\IC^\times.
\end{array}
\right.
\end{equation}
The equation $y^2=x^3+ax+b$ defines an elliptic curve $E$
over $\IC(t)$. We fix an algebraic closure $\overline{\IC(t)}$ of $\IC(t)$.
Lemma \ref{lem:isom} provides $w \in \overline{\IC(t)}^\times$ and
an isomorphism between $E$ 
and an elliptic curve $E'$ given as in (\ref{eq:E2}) with $a',b'\in\IC$.
We regard $E'$ as an elliptic curve defined over $\IC$.
The isomorphism may be taken as an algebraic map  on 
$\pi^{-1}(C)$  with image $E'^3$. We let $A'$ denote the Zariski closure of
the image of $A$ in $E'^3$. 

If $\pi(A)$ is a point, then we take
$A'=A$ regarded as a subvariety of the abelian variety $E'^3=\pi^{-1}(\pi(A))$.

Let   $A$ be an arbitrary irreducible closed subvariety of an
algebraic surface in $\E^3$.
Then $A$ 
is called
 torsion anomalous with respect to the given surface 
\begin{enumerate}
\item [(i)] if $\dim A = 1$ and two independent integral relations hold
  on $A$,  
\item[(ii)] or if 
 $\dim A= 2$ and one non-trivial integral relation holds
  on $A$, 
\item[(iii)]
or  if $\dim A \ge 1$ and 
$j\circ\pi|_A$ is constant and equal to the $j$-invariant of an elliptic
curve with complex multiplication such that, in the notation above,
$A'$ is an
irreducible component of an algebraic subgroup of $E'^3$.
\end{enumerate}

Proposition \ref{prop:finiteness} contained a finiteness statement on
the torsion points outside the torsion
anomalous locus of a surface in $\EL^3$.
The   torsion anomalous
subvarieties of the $123$-surface will cause no problems.

\begin{proposition}
\label{prop:noano}
  The  $123$-surface contains no  torsion anomalous subvarieties.
\end{proposition}

\subsection{Constant $j$-invariant}

As a warm-up for the proof of Theorem \ref{thm:main}
 we show the following weaker version.  An algebraic curve in the
 $123$-surface on which $j$ is constant
 contains only finitely many torsion points. 
We will use this statement in the proof of Proposition \ref{prop:noano}.

 
 \begin{lemma}
\label{lem:constcurve}
   Let $A\subset X$ be an irreducible closed subvariety such that
   $j\circ\pi |_A$ is constant.
    Then $A$ contains only
   finitely many torsion points and $A$ is not a torsion anomalous
   subvariety as in part (iii) of the definition.
 \end{lemma}
 \begin{proof}
We may assume $\dim A \ge 1$.
We remark that $A$ and $\pi(A)=C$ are  algebraic curves since $\pi|_X$ is
dominant and has finite fibers. 

   Let $J\in\IC$ be said $j$-invariant. 
   We let $w,t,A',$ and $E'$ be as in the auxiliary construction before the definition
   of anomalous subvarieties. 
We also consider $a$ and $b$ as elements in $\IC(t)$.

We note that $w\not\in\IC$, indeed, otherwise $a,b$ would be constant
as well
by (\ref{eq:wab}).
Using (\ref{eq:ab})
we find that 
$1+a+b\in \IC(t)$ has odd degree. Therefore, there is a non-trivial
valuation $\ord$ of $\IC(t)$ with $\ord(1+a+b)$ positive and odd.
Using (\ref{eq:ab}) again one finds $\ord(t) = 0$.
Because $a',b'\in\IC$ we can deduce $\ord(w)=0$ from (\ref{eq:wab}).
Therefore, $\IC(w,t)/\IC(t)$ is unramified above $\ord$.
Since  $\{(1,1,1),(8,2,1),(27,3,1)\}$ is linearly dependent
we must have
$\ord(8+2a+b)=0$ or $\ord(27+3a+b)=0$.
 For simplicity say the former
holds; the argument below is readily modified in  the latter case. 
We set $K = \IC(y_2,w,t)$, then $K/\IC(t)$ is unramified above $\ord$.
We 
extend this valuation to $K$ and note that
$K(y_1)/K$ is ramified. Because $y_1^2\in K$ the extension $K(y_1)/K$
 is of degree $2$ and there is an
automorphism $\sigma$ of $K(y_1)/K$ with $\sigma(y_1)=-y_1$.

For $i\in \{1,2,3\}$ we have a point $(i,y_i)\in E(\overline{\IC(t)})$. 
Its image in $E'(\overline{\IC(t)})$ 
under the isomorphism coming from Lemma \ref{lem:isom} is
\begin{equation}
  \label{eq:imy123}
\left(w^2 i-\frac{1}{12},w^3 y_i - \frac 12 w^2 i + \frac{1}{24}\right).
\end{equation}
We may regard $w,y_{1,2,3}$ as rational functions on a ramified
cover of $A$. The three points (\ref{eq:imy123}) determine a rational map
from this cover to $E'^3$. Then  $A'$ is the Zariski closure of its image. 
If $A$ contains infinitely torsion points then so does $A'$.
We use the Manin-Mumford Conjecture for abelian varieties, a result
 first proved by Raynaud \cite{Raynaud:MM}.
It
implies that $A'$ is an irreducible
component of an algebraic subgroup of $E'^3$. 
In particular, there are endomorphisms $\alpha,\beta$ of $E'$, not
both zero,
such that $\alpha(P_1)=\beta(P_2)$ for all $(P_1,P_2,P_3)\in A'$. 
This relation continues to hold generically, i.e. 
$\alpha(P'_1) = \beta(P'_2)$
with $P'_i = (w^2i-1/12,w^3y_i - w^2i/2 + 1/24) \in E(\overline{\IC(t)})$. 
Because $\sigma$ commutes with all endomorphisms of $E'$, which are
defined over $\IC$, we get 
\begin{equation*}
-\beta(P'_2)=-\alpha(P'_1)
=\alpha(w^2-1/12,-w^3y_1-w^2/2+1/24)
= \alpha(P'_1)^{\sigma}
= \beta(P'_2)^{\sigma}=\beta(P'_2).
\end{equation*}
 Therefore, 
$2\beta(P'_2)=0$. 
So one of $P'_1,P'_2\in E'(\overline{\IC(t)})$
 is a torsion point.
But these torsion points are defined over $\IC$ and hence $w\in \IC$. 
This contradicts the fact that $w$ is non-constant.

So $A$ contains only finitely many torsion points. Because an
algebraic subgroup of an abelian variety contains a Zariski dense set
of torsion points we also conclude that $A$ is not torsion anomalous
as in part (iii) of the definition.
 \end{proof}

\subsection{Tate Curves}
In this subsection we collect some basic facts on Tate curves. 
A general reference is Chapter V of Silverman's book
\cite{Silverman:Adv}
or Roquette's book \cite{Roquette}.

Let $K_v$ be a  field, complete with respect to a 
  discrete valuation
$v:K_v\rightarrow \IZ\cup\{+\infty\}$
which we assume to be surjective.
If $q\in K_v^\times$ with $v(q) > 0$ then the  Weierstrass equation
\begin{equation}
\label{eq:tate}
  y^2 + xy = x^3 + a_4(q)x + a_6(q)
\end{equation}
 defines the Tate curve $E_q$ where
 \begin{equation*}
   a_4 = - \sum_{n\ge 1} \frac{n^3 q^n}{1-q^n}\quad\text{and}\quad
a_6 = - \frac{1}{12}\sum_{n\ge 1} \frac{(5n^3+7n^5)q^n}{1-q^n}
 \end{equation*}
converge in $K_v$, cf.
 Theorem V 3.1 \cite{Silverman:Adv}. 
By this theorem and Remark V 3.1.2 {\it ibid.}, cf. Roquette's work
cited above, 
 there exists a surjective homomorphism of groups
\begin{equation*}
  \phi : K_v^\times \rightarrow E_q(K_v)
\end{equation*}
with kernel $q^\IZ$, the infinite cyclic subgroup of $K_v^\times$ generated by  $q$.

We follow  a convenient convention and represent points of
$E_q(K_v)\ssm\{0\}$ using
 affine coordinates. 

Equation (\ref{eq:tate}) has coefficients in
the ring of integers of $K_v$
and is minimal. 
Let $L$ be the residue
field of $K_v$.
The reduction $\widetilde{E_q}$ of $E_q$ is an
irreducible projective curve defined over $L$. We have the reduction
map
$\red : E_q(K_v) \rightarrow \widetilde{E_q}(L)$.
The set of non-singular points of $\widetilde{E_q}(L)$ 
carries a natural abelian group structure.
 We define 
 \begin{equation*}
   E_q(K_v)_0= \{ P\in E_q(K_v);\,\, \red(P)\text{
 is non-singular on }\widetilde{E_q} \}.
 \end{equation*}
This is a subgroup of finite index of $E_q(K_v)$
and $\red|_{E_q(K_v)_0}$ is a homomorphism of groups. 

The Tate uniformization $\phi$ lets us do 
 calculations  explicitly on  Tate curves.

\begin{lemma}
\label{lem:power}
  Let $P\in E_q(K_v)_0\ssm\{0\}$. There is a unique $\tilde u\in K_v$  with
  $v(\tilde u)=0$
and $\phi(\tilde u ) = P$.
Moreover, if $u\in L$ is the reduction of $\tilde u$ then $u\not=0$ and 
  \begin{enumerate}
  \item [(i)] either  $ u=1$ and $\red(P) = 0$,
  \item[(ii)] or $ u \not=1$ and $\red(P) = 
\left( \frac{ u}{(1- u)^2}, *\right)\not=0$. 
  \end{enumerate}
\end{lemma}
\begin{proof}
There is precisely one $ \tilde u\in K_v^\times \ssm q^\IZ$ with $0\le v(\tilde
 u)< v(q)$ and
$\phi(\tilde u) = (x,y)= P$. 
By Lemma V 4.1.1 \cite{Silverman:Adv} we have $v(x) \le 0$ because
$P\in E_q(K_v)_0$; we remark that the proof of this lemma involves only
formal properties of the valuation on $K_v$ and hence holds for any
valued field.

The homomorphism  $\phi$ is explicitly given in Theorem V 3.1 \cite{Silverman:Adv} as
\begin{equation*}
 \phi(\tilde u)= \left(\sum_{n\in\IZ} \frac{q^n \tilde u}{(1-q^n
   \tilde u)^2} - 2 \sum_{n\ge
    1}\frac{nq^n}{1-q^n}, 
 \sum_{n\in\IZ}\frac{q^{2n} \tilde u^2}{(1-q^n\tilde u)^3}+ \sum_{n\ge 1} \frac{n q^n}{1-q^n}
\right)
\end{equation*}
because $\tilde u\not\in q^\IZ$.
 All terms in the sum for $x$  have positive valuation
 except possibly
$\frac{q^n \tilde u}{(1-q^n\tilde u)^2}$ for $n=0$. 
A similar remark holds for $y$. We
 can write
\begin{equation}
\label{eq:NP}
  P=\left(\frac{\tilde u}{(1-\tilde
    u)^2}+x',
\frac{\tilde u^{2}}{(1-\tilde u)^3} +
  y'
  \right)
\quad\text{with}\quad v(x') >0\quad\text{and}\quad v(y')>0.
\end{equation}

Since $v(x) \le 0$ we must
 have
$v(\tilde u) \le 2 v(1-\tilde u)$. This inequality implies
 $v(\tilde u)=0$.
The reduction $u$ of $\tilde u$ is thus non-zero in the residue field $L$.


If $v(1-\tilde u) > 0$, then $u=1$ in $L$.
The orders satisfy
\begin{equation*}
  v(x) = -2 v(1-\tilde u)
\quad\text{and}\quad
  v(y) = -3 v(1-\tilde u).
\end{equation*}
In particular, $y\not=0$ 
and in projective coordinates we have $P = [x/y:1:1/y]$
with $v(x/y) = v(1-\tilde u) > 0$ and $v(1/y) =
3v(1-\tilde u) >
0$.
Therefore,  $\red(P) = 0$ and 
we are in case (i).

On the other hand, if $v(1- \tilde u)\le 0$, then
$v(1-\tilde u)=0$ and so
$ u\not=1$. From (\ref{eq:NP}) we see
that $x$ reduces to $u/(1-u)^2$ in the $L$.
We are in case (ii).
\end{proof}

\subsection{Function Fields}
\label{sec:ff}

Let $K$ be the function field of an  irreducible   algebraic curve
defined over $\IC$. 
Let $a,b\in K$ with $4a^3+27b^2\not=0$. Then 
\begin{equation*}
  y^2 = x^3+ax+b
\end{equation*}
determines an elliptic curve $E$ defined over $K$.

After replacing $K$ by a finite extension  we  have points
\begin{alignat*}1 
P_1 &= (1,*)\in E(K), \quad
P_2 = (2,*)\in E(K), \quad\text{and}\quad
P_3 = (3,*)\in E(K).
\end{alignat*}
The  choice of sign of the second coordinate will be
irrelevant.
After again passing to a finite extension of $K$ we may assume that $E$
has either good or  multiplicative
reduction at all places of $K$.
Multiplicative reduction is automatically split because the residue
field  $\IC$ is algebraically closed.

For any place $v$ of $K$ we let $K_v$ denote the completion of $K$
with respect to $v$. We identify $v$ with the corresponding
 surjective valuation
$K_v\rightarrow \IZ\cup\{+\infty\}$.
We define a finite (possibly empty) 
set
\begin{equation*}
  S= \{\text{places of $K$ where $E$ has bad reduction}\}.
\end{equation*}

If $v\in S$, then 
$E$ is isomorphic over $K_v$ to the Tate curve $E_{q_v}$ for some
$q_v\in K_v^\times$ with $v(q_v) > 0$. Let $f_v:E\rightarrow
E_{q_v}$ be an isomorphism as in Lemma \ref{lem:isom}.
If $v\not\in S$, then $E$ is isomorphic over $K_v$ to an elliptic
curve $E_v$ given by the equation $y^2+xy=x^3+a'x+b'$ with $a',b'$
integers in $K_v$ and with good reduction. 
Let $f_v:E\rightarrow E_{v}$ be an isomorphism given by 
said lemma.
To unify notation we sometimes write $E_v = E_{q_v}$ if $v\in S$.

\begin{lemma}
\label{lem:reduction}
Let $v\in S$ and $i,j\in \{1,2,3\}$ with $i\not=j$. If $f_v(P_i)
  \not\in E_{v}(K_v)_0$, then 
  \begin{equation*}
 \red\, f_v(P_{j}) = \left(\frac{1}{12}\left(\frac
        {j}{i}-1\right),*\right)\in \widetilde{E_v}(\IC)
 \quad\text{and}\quad
f_v(P_j) \in E_{v}(K_v)_0.
  \end{equation*}
\end{lemma}
\begin{proof}
The isomorphism $f_v$ is given on the affine part of $E$ by 
 \begin{equation*}
    (x,y)\mapsto \left(w^2x - \frac{1}{12},*\right).
 \end{equation*}
for some $w\in K_v^\times$.
The reduction $\widetilde{E_{v}}$ is determined by the Weierstrass equation
$y^2+xy=x^3$ and 
 $(0,0)$ is its  only singular point.
By hypothesis, $f_v(P_{i})$ reduces to
$(0,0)$.
Therefore, $v(w^2 i - 1/12) > 0$.
Since $i\not=0$ we find  $v(w^2 - 1/(12 i)) > 0$. In other
words, $w^2$ reduces to $1/(12 i)$ at $v$. So 
$w^2 j-1/12$ reduces to $(j/i-1)/12$ at $v$ and this is the
first coordinate of $\red\, f_v(P_{j})$. 
Finally, because
$i\not=j$ we have
 $f_v(P_{j})\in E_{v}(K_v)_0$.
\end{proof}

For $i\in\{1,2,3\}$ we define the finite (possibly empty) set
\begin{equation*}
  S_i = \{v\in S;\,\, f_v(P_{i}) \not\in E_{v}(K_v)_0\}.
\end{equation*}

For a finite sequence $P,\ldots,Q\in E(K)$ we set
$\rho(P,\ldots,Q)$ to be the rank of the $\IZ$-submodule of $E(K)$ generated by
$P,\ldots, Q$.

\begin{lemma}
\label{lem:admissible}
Suppose $\rho(P_1,P_2,P_3)\le 1$. Then
$S_1=S_2=\emptyset$.
\end{lemma}
\begin{proof}
Say ${i}\in \{1,2\}$.
Assuming the existence of
 $v\in S_i$ we will eventually arrive at a contradiction. 

Let us fix $j$ and $k$ with  $\{i,j,k\}=\{1,2,3\}$ and $j<k$.
By Lemma \ref{lem:reduction} we find
\begin{equation*}
\red\, f_v(P_{j}) = \left(\frac{1}{12} \left(\frac
  {j}{i}-1\right),*\right)\not=0,\quad
  \red\, f_v(P_{k}) = \left(\frac{1}{12} \left(\frac
  {k}{i}-1\right),*\right)\not=0,
\end{equation*}
and $f_v(P_{j}),  f_v(P_{k})\in E_{v}(K_v)_0\ssm\{0\}$.

We apply Lemma \ref{lem:power} to 
$f_v(P_{j})$ and $f_v(P_{k})$ and obtain elements
 $\tilde{u}\in K_v$ and $\tilde{u}'\in K_v$, 
respectively. We are in case (ii) of said
lemma, so $u\not=1$ and $u'\not=1$
for the reductions of $\tilde{u}$ and $\tilde{u}'$,
respectively. These reductions satisfy
\begin{equation}
\label{eq:defu}
 \frac{u}{(1-{u})^2} = \frac{1}{12}\left(\frac{ j}{i}-1\right),\quad\text{and}\quad
\frac{u'}{(1-{u'})^2} = \frac{1}{12}\left(\frac{k}{i}-1\right).
\end{equation}

Since $\rho(P_1,P_2,P_3)\le 1$ 
we have $\rho(P_j,P_k)\le 1$. So
there are $M,N\in\IZ$, not both zero, with
 $[M](P_{j}) = [N](P_{k})$. Using the Tate uniformization, this
relation reads
$\phi({\tilde u}^M) = [M](f_v(P_j)) = [N](f_v(P_k)) = \phi(\tilde{u}'^N)$.
So $\tilde{u}^M \tilde{u}'^{-N} \in q^\IZ$. Since $\tilde{u}$ and
$\tilde{u}'$ 
have valuation zero, we
find
$\tilde{u}^M =\tilde{u}'^N$ and in particular, 
$u^M=u'^N$.

The  contradiction now follows for simply evaluating $u$ and $u'$
in the two possible cases $i=1,2$ using (\ref{eq:defu}). 
Rewriting these identities gives
\begin{alignat*}1
  {u}^2 + 2\frac{5 i+ j}{i-j} {u} +
  1=0 \quad\text{and}\quad
  {u'}^2 + 2\frac{5i+k}{i-k} {u'} + 1=0
\end{alignat*}
with solutions
\begin{equation*}
  (u,u') = 
\left\{
\begin{array}{ll}
  (7\pm 4\sqrt{3},4\pm \sqrt{15} ) &: \text{if }(i,j,k)=(1,2,3), \\
  (-11\pm 2\sqrt{30},13\pm 2\sqrt{42} ) &: \text{if }(i,j,k)=(2,1,3). 
\end{array}
\right.
\end{equation*}

In both cases  $u,u'$ are algebraic
units with $\IQ(u)\cap \IQ(u')=\IQ$.
Hence $u^M=u'^N \in \{\pm 1\}$, the algebraic units of $\IQ$.
So one among $u,u'$ is a root of
unity. 
This is impossible for the totally real $u$ and $u'$;
 the lemma follows.
\end{proof}

One can go a bit further and also show $S_3=\emptyset$. But this will
not be necessary.

If $v$ is any place of $K$, then $\lambda_v$ denotes the  N\'eron local height
on any elliptic curve over $K_v$, cf. Chapter VI
\cite{Silverman:Adv}. 
It does not depend on the choice 
 of a  model of the elliptic curve. 
For a place $v$ of bad reduction we will use the Tate curve
 $E_{q_v}=E_v$  to calculate $\lambda_v$. 
There is an explicit formula for
$\lambda_v$ restricted to $E_{v}(K_v)_0\ssm\{0\}$ given by Theorem VI 4.1
\cite{Silverman:Adv}.
We can use it to handle $\lambda_v(P_1)$ and $\lambda_v(P_2)$
because  $S_1=S_2=\emptyset$.

\begin{lemma}
\label{lem:localheight}
Suppose $\rho(P_1,P_2,P_3)\le 1$.
 If
 $v$ is any place of $K$, then 
\begin{equation*}
 \lambda_v(P_{1}) = \lambda_v(P_{2}).
\end{equation*}
\end{lemma}
\begin{proof}
Let $v$ be any place of $K$.
Recall that $f_v:E\rightarrow E_v$ is an isomorphism of elliptic curves over
$K$.
By Lemma \ref{lem:admissible} the points  $f_v(P_{1})$ and $f_v(P_{2})$
reduce
 to a non-singular point.

Since $P_{1,2}\not=0$
we may use   Theorem VI 4.1 to evaluate
\begin{alignat*}1
  \lambda_v(P_{1}) &= 
  \lambda_v(f_v(P_{1})) = 
   \frac{1}{2} \max\{0,-v(x_1)\}+\frac{1}{12}v(\Delta_v)\quad\text{and} \\
  \lambda_v(P_{2}) &= 
  \lambda_v(f_v(P_{2})) = 
   \frac{1}{2} \max\{0,-v(x_2)\}+\frac{1}{12}v(\Delta_v)
\end{alignat*}
where $x_1$ and $x_{2}$ are the first coordinates of 
$f_v(P_{1})$ and $f_v(P_{2})$, respectively, and $\Delta_v$ is
the local discriminant of $E_{v}$.
We remark that $x_1$ and $x_2$ depend on $v$.

By Lemma \ref{lem:isom}, the isomorphism $f_v$ is determined by some
$w\in K_v^\times$. So
\begin{equation*}
  x_1  = w^2  - \frac{1}{12}\quad\text{and}\quad
  x_{2}  = 2w^2 - \frac{1}{12}.
\end{equation*}
We split up into two cases.

First, let us suppose  $v(w) \ge 0$.
Then $v(x_1)\ge 0$ and $v(x_2)\ge 0$ by the
 ultrametric triangle inequality. So we 
 have
\begin{equation*}
  \lambda_v(P_{1}) = \lambda_v(P_{2}) = 
\frac {1}{12} v(\Delta_v).
\end{equation*}

Second, we assume $v(w) <0$. 
In this case  the ultrametric triangle inequality  yields
$v(x_{1}) = v(x_2)= v(w^2)$. 
Therefore,
\begin{equation*}
  \lambda_v(P_{1}) = \lambda_v(P_{2}) = -\frac 12 v(w^2) +
  \frac{1}{12} v(\Delta_v).\qedhere
\end{equation*}
\end{proof}

Now we will show that $E$ has good reduction everywhere under the
hypothesis of the previous lemma. This is done by a global argument
using  local data from the last lemma. 

The N\'eron-Tate or canonical height is defined for $P\in E(K)\ssm\{0\}$ as 
$\hat h(P) = \sum_v \lambda_v(P)$ where the sum runs over all places
of $K$; for $P=0$ we set $\hat h(P)=0$.

\begin{lemma}
\label{lem:goodreduction}
Suppose
 $\rho(P_{1},P_{2},P_{3})\le 1$.
  Then $S=\emptyset$.
\end{lemma}
\begin{proof}
  First we show that there exists $Q\in E(K)\ssm\{0\}$
with $\hat h(Q)=0$ and $f_v(Q)\in E_{v}(K_v)_0$ for all $v\in S$. 
If $\hat h(P_{1}) = 0$ then we take $Q = P_{1}$ and our claim
follows because $S_1=\emptyset$.
 So say $\hat h(P_{1}) \not= 0$.
By Lemma \ref{lem:localheight} the global heights coincide $\hat
h(P_{1}) = \hat h(P_{2})$.
Since $\rho(P_{1},P_{2})\le 1$ there are
$M,N\in\IZ$ not both zero with $[M](P_{1}) = [N](P_{2})$. 
The N\'eron-Tate height is quadratic, hence $M^2 \hat h(P_{1}) = N^2
\hat h(P_{2}) =  N^2\hat h(P_1)$ and thus $M^2=N^2\not=0$. 
So $[M](P_{1}\pm P_{2})= 0$ and therefore $\hat h(Q) = 0$ 
with $Q=P_{1}\pm P_{2}$. 
Clearly, $Q\not=0$ and
 $f_v(Q)\in
E_{v}(K_v)_0\ssm\{0\}$
 for
all $v\in S$ because $S_1=S_2=\emptyset$.

Now that we have found $Q$ we can easily conclude the proof. Indeed,
the  N\'eron local heights of $Q$  can be evaluated by Theorem VI 4.1.
Just as in the proof of Lemma \ref{lem:localheight},
we use our  model with good reduction $E_v$ if $v\not\in S$ and the Tate curve
$E_{q_v}$ otherwise. 
The N\'eron local heights are non-negative so they  all vanish. But a
N\'eron local height coming from a place of bad reduction contributes by a positive term
through the vanishing order of the local discriminant. 
Therefore, $S=\emptyset$. 
\end{proof}

\begin{lemma}
\label{lem:noanocurve}
  We have $\rho(P_1,P_2,P_3) \ge 2$.
\end{lemma}
\begin{proof}
 We assume $\rho(P_1,P_2,P_3)\le 1$ and deduce a contradiction.

For a certain reordering $(i,j,k)$ of $(1,2,3)$ 
and fixed $M,N,N'\in\IZ$ with $M\not=0$ we have
\begin{equation}
\label{eq:rels}
[N](P_i) = [M](P_j)\quad\text{and}\quad [N'](P_i) = [M](P_k).   
\end{equation}

By the previous lemma we have $S=\emptyset$.
So the $j$-invariant of $E$ is a constant
$2^8 3^3 a^3 /(4a^3+27b^2) \in \IC$.

We may reformulate our situation as follows. There exists an
irreducible algebraic curve $C$ in the $123$-surface for which $j|_{C}$ is constant
and where relations as in (\ref{eq:rels}) hold.

We will  prove below that there are infinitely many points on $C$ where
the $i$-th coordinate is torsion. The relations (\ref{eq:rels}) and Lemma
\ref{lem:constcurve}  lead to a contradiction.

By Lemma \ref{lem:isom} the elliptic curve $E$,
 having constant $j$-invariant, is isomorphic to an elliptic
curve $E'$ given as in (\ref{eq:E2}) with $a',b'\in\IC$. 
This lemma provides an isomorphism
$E\rightarrow E'$
determined by some
 $w\not=0$ in an algebraic closure of $K$. 
We remark $w\not\in \IC$ because $\IC(a,b)$ is not algebraic over $\IC$.
We may regard $w$ as a non-constant algebraic function on $C$.
The image of $P_i$
under this isomorphism is $(w^2i-1/12,*)$. 
We may regard it as an algebraic curve in $E'$.
Now $w^2i-1/12$ attains, up-to finitely many exceptions, 
 any complex value.
In particular, it attains the first coordinate of
a  torsion point of $E'$ infinitely often.
This gives the infinitely many points on $C$ with the desired property.
\end{proof}



\subsection{There are no  Torsion Anomalous
  Subvarieties}
We now prove Proposition \ref{prop:noano}.
First we show that $X$ does not contain any  torsion
anomalous subvarieties as in part (i) of the definition. 
Let $C\subset X$ be an irreducible algebraic curve.
The coordinate functions $a,b:S\rightarrow \IA^1$ induce  rational
functions on $C$.
They determine an elliptic curve $E$ defined over $\IC(a,b)$
given in Weierstrass  $y^2=x^3+ax+b$.
We consider
three points
$P_{1,2,3}$ as in the previous section. Then $\rho(P_1,P_2,P_3)\ge 2$
by Lemma \ref{lem:noanocurve}.
This means that two independent relations cannot simultaneously hold
on $C$. In other words,
 $C$ cannot be  torsion
anomalous. 

Now we show that $X$ cannot contain a torsion anomalous
surface as in part (ii) of the definition. Assuming the contrary, $X$ is a  torsion
anomalous subvariety of itself. 
So there is $(\alpha,\beta,\gamma)\in\IZ^3\ssm\{0\}$ with
\begin{equation*}
  [\alpha] (1,\sqrt{1+a+b}) + [\beta](2,\sqrt{8+2a+b}) +
  [\gamma](3,\sqrt{27+3a+b}) = 0
\end{equation*}
for all $(a,b)\in S$. We suppose first $\beta\not=0$ or
$\gamma\not=0$.
There is an irreducible algebraic curve $C\subset X$ on which $1+a+b=0$
holds identically. So the first coordinate in $\E$ of a point in $C$
has order $2$. In addition to $(\alpha,\beta,\gamma)$, a second and  
independent relation $(2,0,0)$ holds on $C$. Therefore, $C$ is
 torsion anomalous as in part (i) of the definition.
This contradicts the already proven  part of the proposition.
If $\beta=\gamma=0$ we also conclude a contradiction by a similar argument
using a curve on which $8+2a+b=0$ holds.

Finally, by Lemma \ref{lem:constcurve} the surface
 $X$
cannot contain any torsion anomalous subvarieties as in part (iii) of
the definition. \qed

\section{Proof of the Main Result}

Recall that $\mathcal{E}_L$ is the Legendre family of elliptic
curves
over $Y(2)=\IP^1\ssm\{0,1,\infty\}$ and that
 $\mathcal{E}$ is the Weierstrass family of
elliptic curves over $S=\{(a,b);\,\, 4a^3+27b^2\not=0\}$.

Let $X_L$ be an irreducible closed algebraic surface in $\EL^3$. In Section \ref{sec:legendre}
we introduced the notion of a torsion anomalous subvariety of $X_L$.
We call an irreducible closed subvariety of $X_L$ a strongly
torsion anomalous subvariety of $X_L$ if it 
satisfies (i) or (ii)
in the definition of a torsion anomalous subvariety.
 We write $\sta{X_L}$
for $X_L\ssm \bigcup_A A$, here $A$ runs
over all strongly torsion anomalous subvarieties of $X_L$.

Let $X \subset\mathcal{E}^3$ 
be the $123$-surface.
We recall that is irreducible.
  We start off by using it
  to construct an algebraic surface $X_L$ in the Legendre family $\EL^3$.
We first introduce a covering $S'$ of $S$ by setting
\begin{alignat*}1
  S' = \{(a,b,e_1,e_2,e_3,t,r) \in S \times \IA^5;\,\, &e_i^3+a e_i + b
  = 0\text{ for }1\le i\le 3, \\
 &(e_2-e_1)^2(e_3-e_2)^2(e_1-e_3)^2t = 1, \\
 &e_2-e_1 = r^2\}.
\end{alignat*}
Note that $(e_2-e_1)^2(e_3-e_2)^2(e_1-e_3)^2 = -(4a^3+27b^2)$ is
invertible in the coordinate ring of $S$. So $t$ as well as
$e_{1,2,3}$ and $r$  are integral over the
coordinate ring of $S$.
In geometric terms this means that  the natural projection morphism
$S'\rightarrow S$ is finite. 
It is also surjective.
The irreducible components of  $S'$ have dimension $2$.
We obtain a new abelian scheme $\mathcal{E}'^3\rightarrow S'$ by taking the
fibered product of $\mathcal{E}^3\rightarrow S$ with $S'\rightarrow
S$. 
Let $f:\mathcal{E}'^3 \rightarrow \mathcal{E}^3$ be the induced
morphism. It is finite and surjective since these properties are
preserved under base change.
Since $f$ is a closed surjective morphism and $X$ is irreducible, 
the pre-image $f^{-1}(X)$ contains an irreducible component $X'$
with $f(X')=X$. We must have $\dim X' =2$
by standard results in dimension theory,
cf. Exercise II 3.22 \cite{Hartshorne}.

We define a morphism $S'\rightarrow Y(2)$ by 
$(a,b,e_1,e_2,e_3,t,r)\mapsto \frac{e_3-e_1}{e_2-e_1}$.
Then
\begin{equation*}
 (x,y,a,b,e_1,e_2,e_3,r)\mapsto
  \left(\frac{x-e_1}{e_2-e_1},\frac{y}{r^3},\frac{e_3-e_1}{
e_2-e_1}\right)  
\end{equation*}
induces a morphism 
 $g:\mathcal{E}'^3\rightarrow \mathcal{E}_L^3$.
Restricted to a fiber of $\E'^3\rightarrow S'$,
it gives an isomorphism between  Weierstass and Legendre models
of an elliptic curve.
Moreover, it fits into the
 commutative diagram
\begin{equation*}
  \begin{CD}
  \mathcal{E}'^3 @ > g >>  \mathcal{E}_L^3 \\
     @ VVV       @  VVV \\
    S' @>>>  Y(2).
  \end{CD}
\end{equation*}

A straight-forward  verification shows that $g|_{X'}:X'\rightarrow
\mathcal{E}_L^3$ has finite fibers. 
The image $g(X')$ is  constructable  in $\EL^3$ by Chevalley's Theorem.
Let $X_L$ be the Zariski closure of $g(X')$ in $\EL^3$.
Then $X_L$ is irreducible and from
 dimension theory we conclude $\dim X_L = 2$.

Next, let us show
that $X_L$ does not contain any   torsion anomalous subvariety
as in part (ii) of the definition. 
Indeed, otherwise a non-trivial integral relation would hold
identically on
$X_L$. Any such integral relation would hold on $X'$
and also on $X$
 because $g$ is fiberwise 
the cube of an isomorphism of elliptic curves. But no
non-trivial integral relation holds on $X$ by 
Proposition \ref{prop:noano}.

Next, we claim that $X_L$ contains only finitely many    torsion anomalous
subvarieties as in part  (i)  of the definition. 
We also claim that
 each such subvariety intersects $g(X')$ in only finitely many points.
 
Let $C\subset X_L$ be such a  torsion anomalous subvariety.
Then $C$ is an algebraic curve and
there are two possibilities.

Say first that $g|_{X'}^{-1}(C)$ has positive dimension. Then it contains
an irreducible algebraic curve $C'$. Two independent integral relations hold on
$C$. These must continue to hold on $C'$.
Finally, these relations also hold on 
$f(C)\subset X$. Latter must have dimension $1$ because $f$ is a
finite morphism. 
We have found a  torsion anomalous subvariety in $X$ and so a 
 contradiction to Proposition \ref{prop:noano}.

Now say $g|_{X'}^{-1}(C)$ has dimension $0$.
This implies that $C \cap g(X')$ is finite. Because $C$ is irreducible it
follows that $C$ is in 
the Zariski closure of $X_L\ssm g(X')$ in $X_L$.
This closure is a finite union of points and
irreducible algebraic curves. Therefore, it contains $C$ as an irreducible
component. This leaves only finitely many possibilities
for $C$ and our claim  above holds.

 We have proved that
 $g(X') \cap (X_L \ssm \sta{X_L})$ is finite.

Say $P_1,P_2,\ldots$ is a sequence of distinct torsion
points on $X$. We will deduce a contradiction. Since $f|_{X'}:X'\rightarrow X$ is surjective
we find a pre-image, which must be torsion, of each $P_i$  in $X'$.
Because $g|_{X'}$ has finite fibers, $g(X')$
contains infinitely many torsion points $Q_1,Q_2,\dots$. 

By the discussion above, only finitely many of the $Q_1,Q_2,\dots$
can lie on $X_L \ssm \sta{X_L}$.
We remove these from our sequence and suppose $Q_i \in \sta{X_L}$.
By Proposition \ref{prop:finiteness}(i), only
finitely many of the remaining $Q_i$ can lie on $\ta{X_L}$. We remove these as
well. So $Q_i \in \sta{X_L}\ssm \ta{X_L}$.

All $Q_i$ are on an torsion anomalous subvariety of $X_L$ as
 in part (iii) of the definition of torsion
anomalous. 
In particular, each $Q_i$ is  in some  fiber with complex
multiplication. 
We use Proposition \ref{prop:finiteness}(ii).
After 
passing to an infinite subsequence, the $Q_i$
are all in the same fiber of $\EL^3\rightarrow Y(2)$. 
Let $J\in\IC$ be the $j$-invariant of a factor of this fiber.
Each of the corresponding $P_i$ lies in the cube of an elliptic curve
with  $j$-invariant $J$. 
By passing to a  infinite subsequence a last time we find
 infinitely many torsion points  on an
irreducible algebraic curve in $X$ on which the $j$-invariant is
constant. This  contradicts Lemma \ref{lem:constcurve} and completes the
proof of Theorem \ref{thm:main}.
\qed

\bibliographystyle{amsplain}
\bibliography{literature}

\vfill
\address{
\noindent
Philipp Habegger,
Johann Wolfgang Goethe-Universit\"at,
Robert-Mayer-Str. 6-8,
60325 Frankfurt am Main,
Germany,
{\tt habegger@math.uni-frankfurt.de}
}
\bigskip
\hrule
\medskip


\end{document}